\documentclass[11pt,oneside,english]{amsart}
\usepackage[T1]{fontenc}
\usepackage[latin9]{inputenc}
\usepackage{geometry}
\usepackage{textcomp}
\usepackage{mathrsfs}
\usepackage{amsbsy}
\usepackage{amstext}
\usepackage{amsthm}
\usepackage{amssymb}

\makeatletter
\numberwithin{equation}{section}
\numberwithin{figure}{section}
\theoremstyle{plain}
\newtheorem{thm}{\protect\theoremname}[section]
\theoremstyle{plain}
\newtheorem{conjecture}[thm]{\protect\conjecturename}
\theoremstyle{definition}
\newtheorem{example}[thm]{\protect\examplename}
\theoremstyle{remark}
\newtheorem{rem}[thm]{\protect\remarkname}
\theoremstyle{plain}
\newtheorem{prop}[thm]{\protect\propositionname}

\def\makebbb#1{
    \expandafter\gdef\csname#1\endcsname{
        \ensuremath{\Bbb{#1}}}
}\makebbb{R}\makebbb{N}\makebbb{Z}\makebbb{C}\makebbb{H}\makebbb{E}\makebbb{H}\makebbb{P}\makebbb{B}\makebbb{Q}\makebbb{E}\makebbb{E}
\makebbb{H}\makebbb{G}\makebbb{F}\makebbb{D}
\usepackage{babel}

\usepackage{babel}

\makeatother

\usepackage{babel}

\makeatother

\usepackage{babel}
\providecommand{\conjecturename}{Conjecture}
\providecommand{\examplename}{Example}
\providecommand{\propositionname}{Proposition}
\providecommand{\remarkname}{Remark}
\providecommand{\theoremname}{Theorem}

\begin{document}
\title{On K-stability, height bounds and the Manin-Peyre conjecture }
\author{Robert J. Berman}
\dedicatory{In memory of Jean-Pierre Demailly }
\begin{abstract}
This note discusses some intriguing connections between height bounds
on complex K-semistable Fano varieties $X$ and Peyre's conjectural
formula for the density of rational points on $X.$ Relations to an
upper bound for the smallest rational point, proposed by Elsenhans-Jahnel,
are also explored. These relations suggest an analog of the height
inequalities, adapted to the real points, which is established for
the real projective line and related to Kähler-Einstein metrics.
\end{abstract}

\maketitle

\section{Introduction}

The main aim of this note is to discuss some intriguing relations
between bounds on the height of K-semistable complex Fano varieties
and the Manin-Peyre conjecture, concerning the density of rational
points on Fano varieties, that emerged from the work \cite{a-b0}
on height bounds. Along the way a non-technical introduction to the
results in \cite{a-b0,a-b} is provided. Throughout the paper additive
notation for line bundles is employed. In particular, the anti-canonical
line bundle of a non-singular complex projective variety $X$ (i.e.
the top exterior power of the tangent bundle) is denoted by $-K_{X}.$

\subsection{K-stability and height bounds}

Let $X$ be an $n-$dimensional non-singular complex \emph{Fano variety},
i.e. $-K_{X}$ is ample $(-K_{X}>0).$ Since there are only finitely
many families of such varieties their degree, i.e. the top algebraic
intersection number $(-K_{X})^{n}$ of $-K_{X}$, is uniformly bounded
from above by a constant $C_{n}$ only depending on $n.$ For some
time it was expected that the maximal degree is attained when $X=\P^{n}:$
\begin{equation}
(-K_{X})^{n}\leq(-K_{\P^{n}})^{n}\label{eq:fujita intro}
\end{equation}
For example, for $n\leq3$ this follows from the classification of
all $105$ families of non-singular Fano varieties. However, in general,
it turns out that there is not even any polynomial bound on $(-K_{X})^{n}$
for toric Fano manifolds $X$ \cite[page 139]{de0}. Still, as shown
in \cite{ber-ber2} under a symmetry assumption (including the toric
case) and \cite{fu}, in general, the bound \ref{eq:fujita intro}
does hold if one restricts to (possibly singular) \emph{K-semistable
}Fano varieties. Moreover, equality holds only for complex projective
space \cite{liu}. In other words, $\P^{n}$ may be characterized
as the K-semistable Fano variety with the largest degree. In contrast,
for singular Fano varieties the degree can be arbitrarily large in
any given dimension $n$ if K-semistability is not assumed (see \cite[Ex 4.2]{de}
for a simple two-dimensional toric examples). 

The notion of K-stability first arose in the context of the Yau-Tian-Donaldson
conjecture for Fano manifolds, saying that a Fano manifold admits
a \emph{Kähler-Einstein} metric if and only if it is \emph{K-polystable}
\cite{ti,do1}. The conjecture was settled in \cite{c-d-s} and very
recently also established for singular Fano varieties \cite{li1,l-x-z}.
From a purely algebro-geometric perspective K-stability can be viewed
as a limiting form of Chow and Hilbert-Mumford stability that enables
a good theory of moduli spaces (see the survey \cite{x}). 

In \cite{a-b0} a conjectural arithmetic analog of the degree bound
\ref{eq:fujita intro} was proposed. In order to formulate this conjecture
first recall that in the arithmetic setup the Fano variety $X$ is
assumed to be defined over $\Q$. This means that $X$ can be embedded
in a projective space $\P^{m}$ in such a way that $X$ is cut-out
by homogeneous polynomials with rational coefficients. Such an embedding
determines a projective scheme $\mathcal{X}$ over $\Z,$ called a\emph{
model} for $X$ over $\Z.$ We shall call the scheme $\mathcal{X}$
an\emph{ arithmetic Fano variety} if its relative anti-canonical line
bundle over $\Z$ defines as a relatively ample $\Q-$line bundle,
denoted by $-\mathcal{K}_{\mathcal{X}}.$ In this arithmetic setup
there is an analog of the degree called the\emph{ height} \cite{fa,b-g-s}.
This is a real-number, defined with respect to a continuous metric
$\left\Vert \cdot\right\Vert $ on $-K_{X}.$ The height, that we
shall denote by $h_{\left\Vert \cdot\right\Vert }(-\mathcal{K}_{\mathcal{X}})$
is, however, never uniformly bounded from above, as follows readily
from its transformation property under scaling of the metric $\left\Vert \cdot\right\Vert $
(see formula \ref{eq:scaling of chi vol}). Accordingly, in the conjectural
arithmetic analog of the degree bound \ref{eq:fujita intro}, proposed
in \cite{a-b0}, scalings are ruled out by assuming that the metric
$\left\Vert \cdot\right\Vert $ is \emph{volume-normalized} in the
sense that the measure $\mu_{\left\Vert \cdot\right\Vert }$ on $X$
induced the metric on $-K_{X}$ gives total volume one to $X:$ 
\begin{conjecture}
\label{conj:height intro} Let $\mathcal{X}$ be an arithmetic Fano
variety. The following height inequality holds for any volume-normalized
continuous metric on $-K_{X}$ with positive curvature current if
$X$ is K-semistable:
\[
h_{\left\Vert \cdot\right\Vert }(-\mathcal{K}_{\mathcal{X}})\leq h_{\left\Vert \cdot\right\Vert _{\text{FS}}}(-\mathcal{K}_{\P_{\Z}^{n}}),
\]
where $-K_{\P_{\C}^{n}}$ is endowed with the volume normalized Fubini-Study
metric. Moreover, if $\mathcal{X}$ is normal equality holds if and
only if $\mathcal{X}=\P_{\Z}^{n}$ and the metric is Kähler-Einstein,
i.e. coincides with the Fubini-Study metric, modulo the action of
an automorphism.
\end{conjecture}

The conjecture is established for some families of Fano varieties
in \cite{a-b0,a-b} (see Section \ref{sec:K-semistability-and-height}).
The converse ``only if'' statement to the conjecture does hold,
in general. In fact, 
\[
X\,\text{is K-semistable \ensuremath{\text{\ensuremath{\iff} \ensuremath{\sup_{\left\Vert \cdot\right\Vert }h_{\left\Vert \cdot\right\Vert }(\mathcal{L})<\infty} }}}
\]
for any given relatively ample model $(\mathcal{X},\mathcal{L})$
of $(X,-K_{X})$ over $\Z,$ where the sup runs over all volume-normalized
continuous psh metrics on $\mathcal{L}$ (see Theorem \ref{thm:arithm Vol and K semi st}).
However, the upper bound on $h_{\left\Vert \cdot\right\Vert }(\mathcal{L})$
will, in general, depend on the model $\mathcal{L}$ of $-K_{X}$
over $\Z,$unless $\mathcal{L}=-\mathcal{K}_{\mathcal{X}}.$ 

It should be stressed that, in contrast to the degree, the height
$h_{\left\Vert \cdot\right\Vert }(-\mathcal{K}_{\mathcal{X}})$ can
rarely be computed explicitly. But since $h_{\left\Vert \cdot\right\Vert _{\text{FS}}}(-\mathcal{K}_{\P_{\Z}^{n}})$
admits an explicit formula \cite[ §5.4]{g-s},  the previous conjecture
amounts to the following sharp explicit bound on $h_{\left\Vert \cdot\right\Vert }(-\mathcal{K}_{\mathcal{X}}):$

\begin{equation}
h_{\left\Vert \cdot\right\Vert }(-\mathcal{K}_{\mathcal{X}})\leq c_{n},\,\,\,c_{n}:=\frac{1}{2}(n+1)^{n+1}\left((n+1)\sum_{k=1}^{n}k^{-1}-n+\log(\frac{\pi^{n}}{n!})\right).\label{eq:expl formul on p n}
\end{equation}
Equivalently, removing the assumption that $\left\Vert \cdot\right\Vert $
be volume-normalized,
\begin{equation}
\frac{1}{(n+1)}h_{\left\Vert \cdot\right\Vert }(-\mathcal{K}_{\mathcal{X}})+\frac{\text{\ensuremath{(-K_{X})^{n}}}}{2}\log\left(\mu_{\left\Vert \cdot\right\Vert }(X)\right)\leq c_{n}\label{eq:conj with mu volume intro}
\end{equation}

As will be discussed below, this estimate appears to have some intriguing
relations to the Manin-Peyre conjecture in Diophantine geometry, both
as an upper bound on $h_{\left\Vert \cdot\right\Vert }(-\mathcal{K}_{\mathcal{X}}),$
in term of the inverse of $\mu_{\left\Vert \cdot\right\Vert }(X),$
as well as an upper bound on $\mu_{\left\Vert \cdot\right\Vert }(X),$
in terms of $-h_{\left\Vert \cdot\right\Vert }(-\mathcal{K}_{\mathcal{X}}).$ 

Before turning to the Manin-Peyre conjecture, we recall that the height
has an important finiteness property, known as the \emph{Northcott
property: }there is only a finite number of subschemes $\mathcal{X}$
of a given projective space $\P_{\Z}^{m}$ over $\Z$ with uniformly
bounded degree and height, with respect to a fixed metric $\left\Vert \cdot\right\Vert $
on the restriction of the $\mathcal{O}(1)-$line bundle on $\P_{\Z}^{m}:$
\begin{equation}
\sharp\left\{ \mathcal{X}:\,\,\text{\ensuremath{(\mathcal{O}(1)_{|X})^{\dim X}\cdot X\leq d,\,\,\,h_{\left\Vert \cdot\right\Vert }(}}\mathcal{O}(1)_{|\mathcal{X}})\leq B\right\} <\infty.\label{eq:Northcott}
\end{equation}
for given positive numbers $d$ and $B.$ For example, when $\mathcal{X}$
is a hypersurface of a given degree $d$ this follows directly from
the fact that the height $\text{\ensuremath{h_{\left\Vert \cdot\right\Vert }}(}\mathcal{O}(1)_{|\mathcal{X}})$
is comparable to the logarithm of the maximum absolute value $|a_{i}|$
of the integer coefficients $a_{i}$ of the homogeneous polynomial
cutting out $\mathcal{X}$ (see formula \ref{eq:bounds on height of Weil}). 

\subsection{The Manin-Peyre conjecture for the density of rational points}

It is expected that a large class of Fano varieties $X$ defined over
$\Q$ have an abundance of rational points, or more precisely, that
the set of rational points $X(\Q)$ is Zariski dense \cite{b-k-s}.
For example, this is automatically the case for varieties $X$ which
are birationally equivalent to $\P^{n}$ over $\Q.$ In particular,
there is then an infinite number of rational points. In order to quantify
the density of rational points one can first fix a model $(\mathcal{X},\mathcal{L})$
for $(X,-K_{X})$ over $\Z$ and a metric $\left\Vert \cdot\right\Vert $
on $-K_{X}.$ By the Northcott property \ref{eq:Northcott}  the corresponding
number of rational points with exponential height at most $B,$
\begin{equation}
N_{X}(B):=\sharp\left\{ x\in X(\Q):\,\,\text{\ensuremath{H_{\left\Vert \cdot\right\Vert }}(}x)\leq B\right\} ,\,\,\,\text{\ensuremath{H_{\left\Vert \cdot\right\Vert }}(}x):=\exp(\text{\ensuremath{h_{\left\Vert \cdot\right\Vert }}(}x))\label{eq:N X intro}
\end{equation}
is finite for any given real number $B.$ Here $\text{\ensuremath{h_{\left\Vert \cdot\right\Vert }}(}x)$
denotes the height of the scheme over $\Z$ defined by the rational
point $x.$ For example, when $\mathcal{X}$ is a subscheme of $\P_{\Z}^{m}$
and $\mathcal{L}$ is the restriction to $\mathcal{X}$ of $\mathcal{O}(1)\rightarrow\P_{\Z}^{m},$
endowed with the metric induced by a given norm $\left\Vert \cdot\right\Vert $
on $\C^{m+1},$ the height $\text{\ensuremath{H_{\left\Vert \cdot\right\Vert }}(}x)$
may be expressed as
\[
\text{\ensuremath{H_{\left\Vert \cdot\right\Vert }}(}x)=\left\Vert (x_{0},....,x_{m})\right\Vert ,
\]
 when the homogeneous coordinates $x_{i}$ are taken to be integers
with no common divisors (the standard choices for $\left\Vert \cdot\right\Vert $
are the $L^{2}-$norm and the $L^{\infty}-$norms, which correspond
to to the Fubini-Study metric and the Weil metric on $\mathcal{O}(1),$
respectively). Conjecturally, the number $N_{X}(B)$ has almost linear
growth in $B,$ as $B\rightarrow\infty,$ after, perhaps, removing
an appropriate subset $T$ in $X:$ 
\begin{conjecture}
(Manin-Peyre) Assume that $X(\Q)$ is Zariski dense and fix a model
$(\mathcal{X},\mathcal{L})$ for $(X,-K_{X})$ over $\Z$ and a continuous
metric $\left\Vert \cdot\right\Vert $ on $-K_{X}$ with positive
curvature current. Then there exists a thin subset $T$ of $X$ such
that
\[
\lim_{B\rightarrow\infty}\frac{N_{X-T}(B)}{B(\log B)^{r}}=\Theta(\mathcal{L},\left\Vert \cdot\right\Vert )
\]
 for some positive integer $r.$
\end{conjecture}

In the original conjecture of Manin $T$ was assumed to be a subvariety
(accounting for accumulation of rational points), but then counterexamples
were exhibited \cite{b-t}, motivating the more general formulation
above given in \cite{pey2}. While the power $r$ depends only on
$X_{\Q},$ the leading constant $\Theta(\mathcal{L},\left\Vert \cdot\right\Vert )$
depends on the model $(\mathcal{X},\mathcal{L})$ and the metric $\left\Vert \cdot\right\Vert $.
For example, it follows directly from the definitions that, when the
metric $\left\Vert \cdot\right\Vert $ is scaled, so is $\Theta(\mathcal{L},\left\Vert \cdot\right\Vert ).$
An explicit formula for the leading constant $\Theta(\mathcal{L},\left\Vert \cdot\right\Vert )$
was conjectured by Peyre \cite{pey}. For our purposes it will be
enough to recall that the formula may be expressed as a product of
two factors (assuming that ``weak approximation'' holds):
\begin{equation}
\Theta(\mathcal{L},\left\Vert \cdot\right\Vert )=\eta(\mathcal{L})\mu_{\left\Vert \cdot\right\Vert }(X(\R)),\label{eq:decompos of Theta for Q}
\end{equation}
 where $\mu_{\left\Vert \cdot\right\Vert }(X(\R))$ denotes the volume
of the real points $X(\R)$ with respect to the measure $\mu_{\left\Vert \cdot\right\Vert }$
on $X(\R)$ induced by the metric on $\left\Vert \cdot\right\Vert $
(see Section \ref{subsec:Metrics-on-real part}) and $\eta(\mathcal{L}),$
only depends on the model $(\mathcal{X},\mathcal{L})$ of $(X,-K_{X})$
over $\Z.$

Of course, in general $X$ may not admit any rational points. But
it is expected that for any given Fano variety $X$ defined over $\Q$
there exists a finite field extension $\F$ of $\Q$ such that the
points defined over $\F,$ i.e. the $\F-$points $X(\F),$ are Zariski
dense \cite{b-k-s}. In this case the leading constant $\Theta(\F,\mathcal{L},\left\Vert \cdot\right\Vert )$
in the Manin-Peyre conjecture also involves the volume $\mu_{\left\Vert \cdot\right\Vert }(X(\C))$
of the \emph{complex} points $X(\C)$ with respect to the measure
on $X(\C)$ induced by the metric $\left\Vert \cdot\right\Vert $
(see Section \ref{subsec:Motivations-from-the MP }). 

The relation to Conjecture \ref{conj:height intro} stems from its
reformulation \ref{eq:conj with mu volume intro}, involving, in particular,
the complex contribution $\mu_{\left\Vert \cdot\right\Vert }(X(\C))$
to the leading constant $\Theta(\F,\mathcal{L},\left\Vert \cdot\right\Vert )$
in the Manin-Peyre conjecture (which is the only metric contribution
when $\F$ is totally imaginary, e.g. when $\F=\Q+i\Q).$ In Section
\ref{sec:Bounds-on-the minimal} it will be explained that this bound
implies a variant of an upper bound for the height of the smallest
rational point, proposed by Elsenhans-Jahnel \cite{e-j,e-j0} (suggested
by a heuristic argument involving the Manin-Peyre conjecture):
\[
\inf_{x\in X(\bar{\Q})}H_{\phi}(x)\leq\frac{e^{c_{n}/\text{vol}\ensuremath{(-K_{X})}}}{\mu_{\phi}(X(\C))^{1/2}},
\]
where $H_{\phi}(x)$ denotes the exponential absolute height of an
algebraic point $x$ and $X(\bar{\Q})$ denotes the set of all algebraic
points on $X$ (which is always non-empty). In another direction Conjecture
\ref{conj:height intro} implies universal upper bounds on $\mu_{\left\Vert \cdot\right\Vert }(X(\C)).$
For example, when $\left\Vert \cdot\right\Vert $ is \emph{arithmetically
ample} (as in the case of restrictions of the Weil metric or the Fubini-Study
metric) the corresponding height is non-negative $h_{\left\Vert \cdot\right\Vert }(-\mathcal{K}_{\mathcal{X}})\geq0$
(see Section \ref{subsec:Arithmetic-ampleness}). Thus the inequality
\ref{eq:conj with mu volume intro} then implies that 
\[
\mu_{\left\Vert \cdot\right\Vert }(X(\C))^{1/2}\leq\text{\ensuremath{\exp}}\left(\frac{c_{n}}{(-K_{X})^{n}/n!}\right).
\]
 In view of the decomposition \ref{eq:decompos of Theta for Q} of
Peyre's constant for the field $\Q$ one is naturally led to ask if
the K-semistability of $X$ also implies similar upper bounds on $\mu_{\left\Vert \cdot\right\Vert }(X(\R))?$
Some conjectures in this direction are proposed in Section \ref{subsec:The-general-case}
- including the role of Kähler-Einstein metrics - and the case when
$n=1$ is settled. Relations to uniform bounds on Peyre's complete
constant $\Theta(\F,\mathcal{L},\left\Vert \cdot\right\Vert )$ are
also discussed.

\subsection{Acknowledgments}

I am grateful to Rolf Andreasson, Bo Berndtsson, Sébastien Boucksom,
Dennis Eriksson, Gerard Freixas i Montplet, Yuji Odaka and Per Salberger
for illuminating discussions on the topic of this note. Lacking background
in the area of the Manin-Peyre conjecture I apologize for any omission
in accrediting results properly. This work was supported by a grant
from the Knut and Alice Wallenberg foundation.

\section{Setup and background}

\subsection{Complex geometric setup}

Throughout the paper $X$ will denote a compact connected complex
normal variety, assumed to be $\Q-$Gorenstein. This means that the
canonical divisor $K_{X}$ on $X$ is defined as a $\Q-$line bundle:
there exists some positive integer $m$ and a line bundle on $X$
whose restriction to the regular locus $X_{\text{reg }}$ of $X$
coincides with the $m$:th tensor power of $K_{X_{\text{reg }}},$
i.e. the top exterior power of the cotangent bundle of $X_{\text{reg }}.$
We will use additive notation for tensor powers of line bundles. Recall
that the \emph{volume} of a holomorphic line bundle $L\rightarrow X$
is defined by the Hilbert-Samuel formula
\begin{equation}
\text{vol}(L):=\lim_{k\rightarrow\infty}k^{-n}\dim H^{0}(X,kL),\label{eq:HS formully algebraic}
\end{equation}
 which, when $L$ is ample, coincides with $1/n!$ times the top intersection
power $L^{n}$ (also called the \emph{degree }of $X$ with respect
to $L).$ 

\subsubsection{\label{subsec:Metrics-on-line}Metrics on line bundles}

Let $(X,L)$ be a polarized complex projective variety i.e. a complex
normal variety $X$ endowed with an ample holomorphic line bundle
$L.$ We will use additive notation for metrics on $L.$ This means
that we identify a continuous Hermitian metric $\left\Vert \cdot\right\Vert $
on $L$ with a collection of continuous local functions $\phi_{U}$
associated to a given covering of $X$ by open subsets $U$ and trivializing
holomorphic sections $e_{U}$ of $L\rightarrow U:$ 
\begin{equation}
\phi_{U}:=-\log(\left\Vert e_{U}\right\Vert ^{2}),\label{eq:def of phi U}
\end{equation}
 which defines a function on $U.$ Of course, the functions $\phi_{U}$
on $U$ do not glue to define a global function on $X,$ but the current
\[
dd^{c}\phi_{U}:=\frac{i}{2\pi}\partial\bar{\partial}\phi_{U}
\]
 is globally well-defined and coincides with the normalized curvature
current of $\left\Vert \cdot\right\Vert $ (the normalization ensures
that the corresponding cohomology class represents the first Chern
class $c_{1}(L)$ of $L$ in the integral lattice of $H^{2}(X,\R)).$
Accordingly, as is customary, we will symbolically denote by $\phi$
a given continuous Hermitian metric on $L$ and by $dd^{c}\phi$ its
curvature current. The space of all continuous metrics $\phi$ on
$L$ will be denoted by $\mathcal{C}^{0}(L).$ We will denote by $\mathcal{C}^{0}(L)\cap\text{PSH\ensuremath{(L)}}$the
space of all continuous metrics on $L$ whose curvature current is
positive, $dd^{c}\phi\geq0$ (which means that $\phi_{U}$ is plurisubharmonic,
or psh, for short). Then the exterior powers of $dd^{c}\phi$ are
defined using the local pluripotential theory of Bedford-Taylor \cite{b-b}.
Accordingly, the \emph{volume} of an ample line bundle $L$ may be
expressed as
\begin{equation}
\text{vol}(L)=\frac{1}{n!}L^{n}=\frac{1}{n!}\int_{X}(dd^{c}\phi)^{n}\label{eq:vol in terms of phi}
\end{equation}
 where $\phi$ is any element in $\mathcal{C}^{0}(L)\cap\text{PSH\ensuremath{(L)}}.$

More generally, metrics $\phi$ are defined for a $\Q-$line bundle
$L:$ if $mL$ is a bona fide line bundle, for $m\in\Z_{+},$ then
$m\phi$ is a bona fide metric on $mL.$ 

We also recall that a metric $\phi$ on $L$ is psh iff the function
norm on the dual line bundle $-L$ defines a function on the total
space whose logarithm is psh, as illustrated by the following examples: 
\begin{example}
\label{exa:L p metrics}\emph{($L^{p}-$metrics).} Consider $n-$dimensional
complex projective space $\P^{n}$ with homogeneous coordinates $x_{0},...,x_{n}$
and let $L$ be the hyperplane line bundle $\mathcal{O}(1)\rightarrow\P^{n}.$
Its space $H^{0}(\P^{n},\mathcal{O}(1))$ of global holomorphic sections
is spanned by $x_{0},...,x_{n}$ and the complement of the zero-section
in the dual line bundle $-\mathcal{O}(1)\rightarrow\P^{n}$ may be
identified with $\C^{n+1}-\{0\}$ with coordinates $x_{0},...,x_{n}.$
Accordingly, the standard $L^{p}-$norms on $\C^{n+1}$ induce psh
metrics $\phi_{L^{p}}$ on $\mathcal{O}(1)\rightarrow\P^{n}.$ The
metric corresponding to the $L^{2}-$norms is called the \emph{Fubini-Study
metric} and the metric corresponding to the $L^{\infty}-$norm is
called the \emph{Weil metric$.$ }Concretely, letting $U:=\{x_{0}\neq0\}$
be the standard affine piece of $\P^{n}$ with holomorphic coordinates
$z_{i}:=x_{i}/x_{0}$ the function\emph{ $\phi_{L^{p},U}$ }on\emph{
$U$ }representing $\phi_{L^{p}}$ in the standard trivialization
of $\mathcal{O}(1)\rightarrow U,$ defined by the trivializing section
$x_{0},$ is given by 
\[
\phi_{L^{p},U}(z_{1},...,z_{n}):=\log\frac{\left\Vert (x_{0},...,x_{n})\right\Vert _{L^{p}}^{2}}{|x_{0}|^{2}}=\log\left\Vert (1,z_{1}...,z_{n})\right\Vert _{L^{p}}^{2}
\]
More generally, if $kL\rightarrow X$ is very ample one obtains continuous
psh metrics on $L$ by fixing a bases $s_{1},...,s_{N}$ of $H^{0}(X,kL)$
and pulling back the $L^{p}-$norms on $\mathcal{O}(1)\rightarrow\P_{\C}^{N-1}$
under the corresponding projective embedding 
\[
X\hookrightarrow\P^{N-1},\,\,\,x\mapsto[s_{1}(x):...:s_{N}(x)].
\]
\end{example}

\subsubsection{\label{subsec:Metrics-on- minus KX vs volume}Metrics on $-K_{X}$
vs volume forms on $X$}

First consider the case when $X$ is smooth. Then any smooth metric
$\left\Vert \cdot\right\Vert $ on $-K_{X}$ corresponds to a volume
form on $X,$ defined as follows. Given local holomorphic coordinates
$z$ on $U\subset X$ denote by $e_{U}$ the corresponding trivialization
of $-K_{X},$ i.e. $e_{U}=\partial/\partial z_{1}\wedge\cdots\wedge\partial/\partial z_{n}.$
The metric on $-K_{X}$ induces, as in the previous section, a function
$\phi_{U}$ on $U$ and the volume form in question is locally defined
by
\begin{equation}
\mu_{\phi}:=e^{-\phi_{U}}(\frac{i}{2})^{n^{2}}dz\wedge d\bar{z},\,\,\,\,\,dz:=dz_{1}\wedge\cdots\wedge dz_{n},\label{eq:volume form e minus phi U}
\end{equation}
on $U,$ which glues to define a global volume form on $X.$ When
$X$ is singular any continuous metric $\phi$ on $-K_{X}$ induces
a measure on $X,$ defined as before on the regular locus $X_{\text{reg }}$
of $X$ and then extended by zero to all of $X.$ We will say that
a measure $\mu$ on $X$ is a continuous volume form it it corresponds
to a continuous metric on $-K_{X}.$ A Fano variety has log terminal
singularities (in the sense of birational geometry \cite{ko0}) iff
it admits a continuous volume form with finite total volume \cite[Section 3.1]{bbegz}.

\subsubsection{\label{subsec:Metrics-on-real part}Metrics on $-K_{X}$ vs volume
forms on $X(\R)$ }

Now assume that $X$ is a complex manifold endowed with a real structure,
i.e. an anti-holomorphic involution $\sigma.$ The case that we shall
be interested in here is when $X$ is cut out by homogeneous polynomials
on $\P_{\C}^{m}$ with real coefficients and $\sigma$ is the restriction
to $X$ of complex conjugation on $\P_{\C}^{m}.$ Denote by $X(\R)$
the corresponding fix point locus in $X.$ A metric $\phi$ on $-K_{X}$
(or more precisely, on the restriction of $-K_{X}$ to $X(\R)$ )
induces a volume form on $X(\R)$ that we shall also denote by $\mu_{\phi}$
(abusing notation, slightly) : 
\[
\mu_{\phi}:=e^{-\frac{1}{2}\phi_{U}}dx,\,\,\,\,dx:=dx_{1}\wedge\cdots\wedge dx_{n},
\]
 where we have fixed local holomorphic coordinates $z_{i}$ such that
$\sigma z_{i}=\bar{z}_{i}$ and $x_{i}$ denotes the real part of
$z_{i}.$ Just as in the complex setup it readily follows that $\mu_{\phi}$
define a global volume form on $X(\R).$

\subsection{Kähler-Einstein metrics and K-stability}

\subsubsection{Kähler-Einstein metrics }

A Kähler metric $\omega$ on a complex manifold $X$ is said to be
\emph{Kähler-Einstein} if its has constant Ricci curvature. When $X$
is Fano, i.e. $-K_{X}>0,$ then a Kähler-Einstein metric on $X$ necessarily
has positive Ricci curvature. This means that the the metric $\phi$
on $-K_{X}$ corresponding to the volume form $\omega^{n}/n!$ has
positive curvature. The Kähler-Einstein condition on $\omega$ is
equivalent to the following Monge-Ampère equation for a volume-normalized
psh metric $\phi$ on $-K_{X}:$
\begin{equation}
\frac{1}{\ensuremath{(-K_{X})}^{n}}(dd^{c}\phi)^{n}=\mu_{\phi}\text{ \,\,on \ensuremath{X}.}\label{eq:ke eq}
\end{equation}
By the solution of the Yau-Tian-Donaldson conjecture for Fano manifolds
\cite{c-d-s} a Fano manifold $X$ admits a Kähler-Einstein metric
iff $(X,-K_{X})$ is\emph{ K-polystable. }Very recently the conjecture
was also settled for singular Fano varieties $X$ \cite{li1,l-x-z}
(in the singular case a Kähler-Einstein metric $\phi$ is defined
as a locally bounded psh metric on $-K_{X}$ whose curvature form
defines a smooth Kähler-Einstein metric on the regular locus of $X).$
The proof in the singular case builds on the variational approach
to the Yau-Tian-Donaldson conjecture introduced in \cite{b-b-j},
whose starting point is the fact that the Kähler-Einstein equation
\ref{eq:ke eq} is the critical point equation for the \emph{Ding
functional} $\mathcal{D}_{\phi_{0}}(\phi)$ defined by 
\begin{equation}
\mathcal{D}_{\phi_{0}}(\phi):=-\mathcal{E}_{\phi}(\phi)-(-K_{X})^{n}\log\int_{X}\mu_{\phi},\label{eq:Ding}
\end{equation}
 where $\phi_{0}$ is a fixed reference metric on $-K_{X}$ and 
\begin{equation}
\mathcal{E}_{\phi_{0}}(\phi):=\frac{1}{(n+1)!}\int_{X}(\phi-\phi_{0})\sum_{j=0}^{n}(dd^{c}\phi)^{n-j}\wedge(dd^{c}\phi_{0})^{n-j}.\label{eq:def of functional E}
\end{equation}
 This follows directly from the fact the functional $\mathcal{E}_{\phi_{0}}(\phi)$
is a primitive of the Monge-Ampère operator, i.e its differential
at a given metric $\phi$ is represented by the Monge-Ampère measure:
\begin{equation}
d\mathcal{E}_{\phi_{0}}(\phi)=\frac{(dd^{c}\phi)^{n}}{n!}.\label{eq:diff of functional E}
\end{equation}

\subsubsection{K-stability}

We briefly recall the notion of K-semistability (see the survey \cite{x}
for more background). A polarized complex projective variety $(X,L)$
is said to be \emph{K-semistable} if the Donaldson-Futaki invariant
$\text{DF}(\mathscr{X},\mathscr{L})$ of any test configuration $(\mathscr{X},\mathscr{L})$
for $(X,L)$ is non-negative, $\text{DF}(\mathscr{X},\mathscr{L})\geq0.$
A test configuration $(\mathscr{X},\mathscr{L})$ is defined as a
$\C^{*}-$equivariant normal model for $(X,L)$ over the complex affine
line $\C.$ More precisely, $\mathscr{X}$ is a normal complex variety
endowed with a $\C^{*}-$action $\rho$, a $\C^{*}-$equivariant holomorphic
projection $\pi$ to $\C$ and a relatively ample $\C^{*}-$equivariant
$\Q-$line bundle $\mathscr{L}$ (endowed with a lift of $\rho):$
\begin{equation}
\pi:\mathcal{\mathscr{X}}\rightarrow\C,\,\,\,\,\,\mathscr{L}\rightarrow\mathscr{X},\,\,\,\,\,\,\rho:\,\,\mathscr{X}\times\C^{*}\rightarrow\mathscr{X}\label{eq:def of pi for test c}
\end{equation}
such that the fiber of $\mathscr{X}$ over $1\in\C$ is equal to $(X,L).$
Its \emph{Donaldson-Futaki invariant} $\text{DF}(\mathscr{X},\mathscr{L})\in\R$
may be defined as a normalized limit, as $k\rightarrow\infty,$ of
Chow weights of a sequence of one-parameter subgroups of $GL\left(H^{0}(X,kL)\right)$
induced by $(\mathscr{X},\mathscr{L})$ (in the sense of Geometric
Invariant Theory). Alternatively, according to the intersection-theoretic
formula for $\text{DF}(\mathscr{X},\mathscr{L})$ established in \cite{w,od1}, 

\[
\text{DF}(\mathscr{X},\mathscr{L})=\frac{a}{(n+1)!}\overline{\mathscr{L}}^{n+1}+\frac{1}{n!}\mathscr{K}_{\mathcal{\mathscr{\overline{X}}}/\P^{1}}\cdot\mathcal{\overline{\mathscr{L}}}^{n},\,\,\,\,a=-n(K_{X}\cdot L^{n-1})/L^{n}
\]
 where $\overline{\mathscr{L}}$ denotes the $\C^{*}-$equivariant
extension of $\mathscr{L}$ to the $\C^{*}-$equivariant compactification
$\mathscr{\overline{X}}$ of $\mathscr{X}$ over $\P^{1}$ and $\mathscr{K}_{\mathcal{\mathscr{\overline{X}}}/\P^{1}}$
denotes the relative canonical divisor. A polarized variety $(X,L)$
is said to be \emph{K-polystable }if $\text{DF}(\mathscr{X},\mathscr{L})\geq0$
with equality iff $(\mathscr{X},\mathscr{L})$ is isomorphic to a
product and \emph{K-stable} if equality only holds when  $(\mathscr{X},\mathscr{L})$
is isomorphic to a product in a $\C^{*}-$equivariant way. 
\begin{example}
\label{exa:K-st}Let $X$ be a Fano hypersurface in $\P^{n+1},$ i.e.
$X$ is cut-out by a homogeneous polynomial of degree $d$ at most
$n+1.$When $d\leq2$ $(X,-K_{X})$ is K-polystable (since $X$ is
homogeneous) and for $d\geq3$ it is conjectured that $(X,-K_{X})$
is K-stable. The case $d\geq n$ is established in \cite{x}. Moreover,
for any $d$ the corresponding Fermat hypersurface is K-stable \cite{zhu},
which implies that a generic hypersurface of degree $d$ is K-stable
(see \cite{x}). In the toric case a Fano variety $(X,-K_{X})$ is
K-polystable iff the moment polytope $P$ of $(X,-K_{X})$ has its
barycenter at the origin \cite{ber0}.
\end{example}

\subsection{Arithmetic setup}

\subsubsection{\label{subsec:Models-over}Models over $\Z$}

Let now $X$ be complex projective variety, which is defined over\emph{
$\Q.$ }Concretely, this means that $X$ can be embedded in a projective
space $\P^{m}$ in such a way that $X$ is cut out by a finite number
$r$ of homogeneous polynomials $f_{i}$ with rational coefficients.
Such an embedding determines a projective flat scheme $\mathcal{X}$
over $\Z,$ which is called a \emph{model for $X$ over $\Z.$} The
scheme $\mathcal{X}$ may be defined as\emph{ }$\text{Proj \ensuremath{\left(\Z[x_{0},...,x_{m}]/I\right),}}$
where $I$ denotes the ideal obtained as the intersection of $\Z[x_{0},...,x_{m}]$
with the ideal in $\Q[x_{0},...,x_{m}]$ generated by $f_{1},f_{2},...$\cite[II.2]{ha}.
An embedding of $X$ as above also induces a model $\mathcal{L}$
over $\Z$ of the line bundle $L$ on $X$, defined as the restriction
to $X$ of the $\mathcal{O}(1)-$line bundle on $\P_{\C}^{m}.$ For
any positive integer $k$ we may identify the free $\Z-$module $H^{0}(\mathcal{X},k\mathcal{L})$
of its global sections with a lattice in $H^{0}(X,kL):$
\[
H^{0}(\mathcal{X},k\mathcal{L})\otimes\C=H^{0}(X,kL).
\]
Concretely, for $k$ sufficiently large the lattice in $H^{0}(X,kL)$
consists of restrictions to $X$ of polynomials with integer coefficients.

\subsubsection{Heights}

The height of a a line bundle $\mathcal{L}\rightarrow\mathcal{X}$
over $\Z,$ with respect to a given continuous psh metric $\phi$
on the complexification $L\rightarrow X,$ may be defined by the arithmetic
Hilbert-Samuel formula \cite{Zh0}
\begin{equation}
h_{\phi}(\mathcal{L}):=(n+1)!\lim_{k\rightarrow\infty}\log\text{Vol}\text{\ensuremath{\left\{  s_{k}\in H^{0}(\mathcal{X},k\mathcal{L})\otimes\R:\,\,\,\sup_{X}\left\Vert s_{k}\right\Vert _{\phi}\leq1\right\} } , }\label{eq:def of xhi L infty}
\end{equation}
where $\text{Vol}$ i the unique Haar measure on $H^{0}(\mathcal{X},k\mathcal{L})\otimes\R$
which gives unit-volume to a fundamental domain of the lattice $H^{0}(\mathcal{X},k\mathcal{L}).$ 
\begin{rem}
Originally, the height was defined in \cite{fa} using arithmetic
intersection theory \cite{g-s,g-s2}, in the context of Arakelov geometry,
as the top arithmetic intersection number of the metrized line bundle
$(\mathcal{L},\phi)$ on the $(n+1)-$dimensional scheme $\mathcal{X}:$
\[
h_{\phi}(\mathcal{L}):=(\mathcal{L},\phi)^{n+1},
\]
 The metric $\left\Vert \cdot\right\Vert $ on $L\rightarrow X$ thus
plays the role of a ``compactification'' of $\mathcal{X}.$ But
for the purpose of the present note it will be convenient to adopt
the direct definition \ref{eq:def of xhi L infty}.
\end{rem}

More generally, the height is naturally defined for $\Q-$line bundles,
since it is homogeneous with respect to tensor products$:$ 
\begin{equation}
h_{\phi}(m\mathcal{L})=m^{n+1}h_{\phi}(\mathcal{L}),\,\,\,\text{if \ensuremath{m\in\Z_{+}}}\label{eq:vol chi tensor}
\end{equation}
Moreover, the \emph{normalized height }
\[
\hat{h}_{\phi}(\mathcal{L}):=\frac{h_{\phi}(\mathcal{L})}{(n+1)!\text{vol \ensuremath{(L)}}}
\]
\emph{ }is additively equivariant with respect to scalings of the
metric: 
\begin{equation}
\hat{h}_{\phi+\lambda}(\mathcal{L})=\hat{h}_{\phi}(\mathcal{L})+\frac{\lambda}{2},\,\,\,\text{if }\lambda\in\R,\label{eq:scaling of chi vol}
\end{equation}
 as follows directly from the definition, using the classical Hilbert-Samuel
formula$.$ We also recall that, given two continuous psh metrics
$\phi$ and $\phi_{0}$ on the complexification $L\rightarrow X$
of $\mathcal{L}\rightarrow\mathcal{X},$ the following ``change of
metrics formula'' holds:
\begin{equation}
h_{\phi}(\mathcal{L})-h_{\phi_{0}}(\mathcal{L})=\frac{(n+1)!}{2}\mathcal{E}_{\phi_{0}}(\phi),\label{eq:change of metric formula for chi vol}
\end{equation}
 where $\mathcal{E}_{\phi_{0}}(\phi)$ is the primitive of the Monge-Ampère
measure defined by formula \ref{eq:diff of functional E} (this is
a special case of \cite[Thm A]{b-b}). 
\begin{rem}
By formula \ref{eq:change of metric formula for chi vol} the functional
$\phi\mapsto h_{\phi}(\mathcal{L})$ is also a primitive of Monge-Ampère
operator, for any given model $(\mathcal{X},\mathcal{L})$ of $(X,L)$
over $\Z.$ Thus, loosely speaking, the model over $\Z$ plays the
role of the reference metric $\phi_{0}.$ 
\end{rem}

\begin{example}
\emph{\label{exa:(height-of-weil rational}(height of a rational point).}
Let $x$ be a rational point on $X,$ i.e. point in $X(\Q).$ A fixed
model $(\mathcal{X},\mathcal{L}$) over $\Z$ induces, by restriction,
a model over $\Z$ for the the point $x$ and it follows directly
from the definition \ref{eq:def of xhi L infty} that
\[
h_{\phi}(x):=-\log\left\Vert s_{x}\right\Vert _{\phi},
\]
where $s_{x}$ is a generator of the free rank one $\Z-$module in
the complex line $L_{|x}$ induced by $(\mathcal{X},\mathcal{L}).$
Concretely, if $(\mathcal{X},\mathcal{L}$) is the model over $\Z$
induced by a projective embedding of $X,$ then, representing $x=[x_{0}:....:x_{m}]\in\P^{m}$
for $x_{i}$ integers with no common divisors we can express
\[
h_{\phi}(x)=\log\left\Vert (x_{0},...,x_{m})\right\Vert ,
\]
where now $\left\Vert \cdot\right\Vert $ denotes the induced metric
on the dual line bundle $-L$ (using that, away from the zero-section,
$-L$ may be identified with a subvariety of $\C^{m+1}-\{0\}).$ In
the case when $\left\Vert \cdot\right\Vert $ is the restriction to
$X$ of the \emph{Weil metric} on $\P^{m}$ (see Example \ref{exa:L p metrics})
the corresponding height on points is called the ``naive height''
(see Example \ref{exa:(Naive-heights).-Let}). Note that ``naive''
height of $H_{\phi}(x)$ is a measure of the information-theoretic
complexity of $x.$ For example, $h_{\phi}(1:\frac{1}{2})=2,$ but
$h_{\phi}(1:\frac{1000}{1999})=1999,$ although the points $(1:\frac{1}{2})$
and $(1:\frac{1000}{1999})$ on $\P^{1}$ are close to each other
in the metric sense. Moreover, in general, $h_{\phi}(1:1/m)=h_{\phi}(1:m)$
for any integer $m.$
\end{example}

In the case of projective space the height may be analytically expressed
by the following well-known formula:
\begin{equation}
2h_{\phi}(\P_{\Z}^{n},\mathcal{O}(1))=(n+1)!\mathcal{E}_{\P^{n+1}}(\phi,\phi_{0}),\label{eq:height on Pn}
\end{equation}
 where $\mathcal{E}_{\P^{n+1}}$ denotes the functional defined by
formula \ref{eq:def of functional E} corresponding to $\mathcal{O}(1)\rightarrow\P^{n}$
and $\phi_{0}$ is the Weil metric on $\mathcal{O}(1)$ (for example,
this is a special case of the toric formula in \cite[formula 3.7]{a-b0}).
More generally, the height of a hypersurface in $\P^{n+1}$ of degree
$d$ may be analytically expressed using the restriction formula for
the height \cite[Prop 2.3.1]{b-g-s}:
\begin{example}
\emph{\label{exa:(height-of-a}(height of a hypersurface)} Let $\mathcal{X}$
be the subscheme of $\P_{\Z}^{n+1}$ cut out by a homogeneous polynomial
$s$ of degree $d$ with integer coefficients and $\phi$ a continuous
psh metric on $\mathcal{O}(d)\rightarrow\P_{\C}^{n+1}.$ Then the
height $h_{\phi}(\mathcal{X}_{d},\mathcal{O}(d))$ of the restriction
of $(\mathcal{O}(d),\phi)$ to $\mathcal{X}$ may be expressed as
\[
\frac{2h_{\phi}(\mathcal{X}_{d},\mathcal{O}(d))}{(n+1)!}=(n+2)\mathcal{E}_{\P^{n+1}}(\phi,d\phi_{0})+\int_{\P^{n+1}}\log\left(\left\Vert s\right\Vert _{\phi}^{2}\right)\frac{(dd^{c}\phi)^{n+1}}{(n+1)!},
\]
 where $\phi_{0}$ is the Weil metric on $\mathcal{O}(1)$ and $\mathcal{E}_{\P^{n+1}}$
denotes the functional defined by formula \ref{eq:def of functional E},
corresponding to $\mathcal{O}(d)\rightarrow\P^{n+1}.$

In particular, when $\phi$ is the metric on $\mathcal{O}(d)$ induced
from the Weil metric on $\P^{n+1}$ the height $h_{\phi}(\mathcal{X}_{d},\mathcal{O}(d))$
is given by the\emph{ Mahler measure} \cite{mah2} of $s:$
\[
h_{\phi}(\mathcal{X}_{d},\mathcal{O}(d))=\int_{T^{n+1}}\log\left|s\right|d\theta/(2\pi)^{n},
\]
 where $T^{n+1}$ is the unit-torus in the affine piece of $\P^{n+1}$
endowed with its Haar measure $d\theta/(2\pi)^{n}.$ Denoting by $a_{i}$
the coefficients in $\Z$ of $s$ it thus follows from the estimates
in \cite{mah2} that 
\begin{equation}
\left|h_{\phi}(\mathcal{X}_{d},\mathcal{O}(d))-\max_{i}|\text{\ensuremath{a_{i}}}|\right|\leq b(n,d)\label{eq:bounds on height of Weil}
\end{equation}
 for an (explicit) constant $b(n,d)$ only depending on $n$ and $d.$ 
\end{example}

\subsubsection{\label{subsec:Arithmetic-ampleness}Arithmetic ampleness}

A metrized line bundle $(\mathcal{L},\phi)$ is said to be \emph{arithmetically
effective} \cite{Zh0} if, for $k$ sufficiently large, there exists
a basis $s_{1},...,s_{N_{k}}$ in $H^{0}(\mathcal{X},k\mathcal{L})$
satisfying $\left\Vert s_{i}(x)\right\Vert _{\phi}\leq1$ at all points
$x$ of $X.$ It then follows directly from the definition \ref{eq:def of xhi L infty}
that 
\[
h_{\phi}(\mathcal{L})\geq0.
\]
 For example, $(\mathcal{L},\phi)$ is arithmetically ample when $\mathcal{L}$
is the restriction to $\mathcal{X}$ of $\mathcal{O}(1)\rightarrow\P_{\Z}^{m}$
and $\phi$ is the restriction of either the Weil metric or the Fubini-Study
metric. 

\section{\label{sec:K-semistability-and-height}K-semistability and height
bounds }

In this section we recall some of the results in \cite{a-b0,a-b},
linking K-semistability to height bounds (see also \cite{o} for other
relations between K-semistability and heights for general polarized
manifolds $(X,L))$. The starting point  is the following result (see
\cite[Thm 2.4]{a-b} for a more general formulation):
\begin{thm}
\label{thm:arithm Vol and K semi st}Let $X$ be a Fano variety and
fix a relatively ample model $(\mathcal{X},\mathcal{L})$ for $(X,-K_{X})$
over $\Z.$ Then the following is equivalent:
\begin{itemize}
\item $(X,-K_{X})$ is K-semistable
\item The supremum of the height $h_{\phi}(\mathcal{L})$ over all volume-normalized
continuous psh metrics $\phi$ on $-K_{X}$ is finite.
\end{itemize}
\end{thm}

\begin{proof}
\emph{(sketch). }In the case when $X$ is non-singular the proof simply
amounts to an arithmetic reformulation of the main result in \cite{li}.
Indeed, by the scaling relation \ref{eq:scaling of chi vol}, the
finiteness in the second point of the previous theorem is equivalent
to the lower boundedness of the functional 
\begin{equation}
\mathcal{D}_{\Z}(\phi):=-2h_{\phi}(\mathcal{L})/(n+1)!-\text{vol\ensuremath{(-K_{X})}}\log\int_{X}\mu_{\phi}.\label{eq:arithm Ding}
\end{equation}
This functional coincides, up to an additive constant, with the \emph{Ding
functional }\ref{eq:Ding} (as follows directly from formula \ref{eq:change of metric formula for chi vol}).
As shown in \cite{li} $(X,-K_{X})$ is K-semistable iff the Ding
functional is bounded from below. The ``only if'' direction was
shown in \cite{ber0} ( the general singular setup). The starting
point of the proof of the ``if direction'' in\cite{li}, is the
solution of the Yau-Tian-Donaldson conjecture for Fano manifolds in
\cite{c-d-s}, which implies that that $(X,-K_{X}$) is K-polystable
iff the Ding functional admits a minimizer, namely a Kähler-Einstein
metric. However, since $(X,-K_{X})$ is merely assumed K-semistable
the Kähler-Einstein metric lives on a deformation $X_{0}$ of $X$
(defined as the central fiber of a special test configuration with
vanishing Donaldson-Futaki invariant). As shown in \cite{li} this
is enough to conclude that that $\mathcal{D}(\phi)$ is bounded from
below (although its infimum is not attained, unless $(X,-K_{X})$
is K-polystable). In \cite{a-b0} this result is extended to singular
Fano varieties, using, in particular, the solution of the Yau-Tian-Donaldson
conjecture for singular Fano varieties \cite{li1,l-x-z}. 
\end{proof}
One virtue of the arithmetic setup is that the arithmetic Ding functional
$\mathcal{D}_{\Z}$ in formula \ref{eq:arithm Ding} is independent
of the choice of a reference metric. Indeed, it only depend on the
choice of a model $(\mathcal{X},\mathcal{L})$ of $(X,-K_{X})$ over
$\Z.$ In particular, if $\mathcal{L}$ is the relative anti-canonical
line bundle $-\mathcal{K}_{\mathcal{X}}$ of $\mathcal{X}$, then
the infimum of the arithmetic Ding functional is a real number only
depending on the integral model $\mathcal{X}$ of $X$ over $\Z.$
Conjecture \ref{conj:height intro} may thus be reformulated as the
following
\begin{conjecture}
\label{conj:height not intro}Let $\mathcal{X}$ be a projective scheme
over $\Z$ of relative dimension $n$ whose relative anti-canonical
divisor defines a relatively ample $\Q-$line bundle $-\mathcal{K}_{\mathcal{X}}.$
Assume that $(X,-K_{X})$ is K-semistable. Then the following numerical
invariant of $\mathcal{X}$ 
\[
\sup_{\phi}\left(h_{\phi}(-\mathcal{K}_{\mathcal{X}})/(n+1)!+\frac{1}{2}\text{vol\ensuremath{(-K_{X})}}\log\int_{X}\mu_{\phi}\right),
\]
 (where $\phi$ ranges over all continuous psh metrics on $-K_{X}$)
is maximal when $\mathcal{X}=\P_{\Z}^{n}.$ Moreover, if $\mathcal{X}$
is assumed normal $\P_{\Z}^{n}$ is the only maximizer of the invariant
above (and the sup over $\phi$ is then attained precisely when $\phi$
is a Kähler-Einstein metric).
\end{conjecture}

We will say that a \emph{weak version} of the previous conjecture
holds if there exists a constant $C_{n},$ only depending the dimension
$n,$ such that 
\[
\frac{h_{\phi}(-\mathcal{K}_{\mathcal{X}})}{(n+1)!}+\frac{\text{vol}\ensuremath{(-K_{X})}}{2}\log\int_{X}\mu_{\phi}\leq C_{n}
\]
(see \cite{na} for a similar conjecture).
\begin{rem}
The invariant in Conjecture \ref{conj:height not intro} is reminiscent
of Bost and Zhang's intrinsic height \cite{bo1,bo2,zh1}, which, in
particular, applies to $(\mathcal{X},-k\mathcal{K}_{\mathcal{X}}),$
when $-k\mathcal{K}_{\mathcal{X}}$ is very ample. The intrinstic
height of $(\mathcal{X},-k\mathcal{K}_{\mathcal{X}})$ is finite iff
$(\mathcal{X},-k\mathcal{K}_{\mathcal{X}})$ is Chow semistable. However,
it is defined as an \emph{infimum} involving $h_{\phi}(-\mathcal{K}_{\mathcal{X}})$
(where $\log\int_{X}\mu_{\phi}$ is replaced by another term) and
thus yield\emph{ lower} bounds on the heights $h_{\phi}(-k\mathcal{K}_{\mathcal{X}})$
(see the discussion in \cite[Section 6.3]{a-b0}). 
\end{rem}

\subsection{The toric case}

Let let $X$ be a toric Fano variety. The main result of \cite{a-b0}
establishes, in particular, the previous conjecture for the canonical
integral model $\mathcal{X}$ of $X$ over $\Z$ when $n\leq6.$ To
explain this result denote by $P$ the moment polytope in $\R^{n}$
of the polarized toric variety $(X,-K_{X}).$ The canonical integral
model $\mathcal{X}$ may be concretely described as follows. Denote
by $m_{1},...,m_{N}$ the vertices of the scaled polytope $kP$ for
a sufficiently large integer $k.$ Then $X$ may be identified with
the Zariski closure in $\P^{N-1}$ of the image of the multinomial
map 
\begin{equation}
\C^{*n}\rightarrow\P^{N-1},\,\,\,z\mapsto[z^{m_{1}}:z^{m_{2}}:....:z^{m_{N}}],\label{eq:monomial em}
\end{equation}
The image is defined over $\Q$ and thus determines an integral model
$\mathcal{X}$ of $X.$ Concretely, the defining polynomials may,
in this toric case, be taken to be \emph{binomials}. More precisely,
the homogeneous polynomials $f$ cutting out $\mathcal{X}$ may be
taken to be of the form 
\[
f=x^{\alpha_{+}}-x^{\alpha_{-}},\,\,\,\,\alpha=(a_{1},...,a_{N})\in\Z^{N}:\,\,\,\sum_{i=1}^{N}a_{i}m_{i}=0
\]
(see \cite{c-l-s}). Moreover, $-\mathcal{K}_{\mathcal{X}}$ is defined
as a relatively ample $\Q-$line bundle ($-k\mathcal{K}_{\mathcal{X}}$
is isomorphic to the restriction of $\mathcal{O}(1)\rightarrow\P_{\Z}^{N-1}$
to $\mathcal{X}$ under the embedding \ref{eq:monomial em}; see \cite{a-b0}).
\begin{thm}
\label{thm:main toric intro}Let $X$ be an $n-$dimensional K-semistable
toric Fano variety and denote by $\mathcal{X}$ the canonical model
of $X$ over $\Z.$ Then Conjecture \ref{conj:height not intro} holds
if $n\leq6$ and $X$ is $\Q-$factorial (equivalently, $X$ is non-singular
or has abelian quotient singularities). 
\end{thm}

\begin{proof}
\emph{(sketch) }The first step of the proof is to establish following
bound
\begin{equation}
\frac{h_{\phi}(-\mathcal{K}_{\mathcal{X}})}{(n+1)!}\leq-\frac{1}{2}\text{vol}(-K_{X})\log\left(\frac{\text{vol}(-K_{X})}{(2\pi^{2})^{n}}\right)\label{eq:universal bound intro}
\end{equation}
Since $\text{vol}(-K_{X})$ is maximal for $X=\P^{n}$ \cite{fu}
the right hand side above is bounded by a constant $C_{n}$ only depending
on the dimension $n$ (in particular, this shows that the weak form
of Conjecture \ref{conj:height not intro} holds for toric varieties
in any dimension). Under the ``gap hypothesis'' that $\P^{n-1}\times\P^{1}$
has the second largest volume among all $n-$dimensional K-semistable
$X$ we show that the bound \ref{eq:universal bound intro} implies
Conjecture \ref{conj:height intro} for the canonical integral model
$\mathcal{X}$ of a toric Fano variety $X$ (by replacing $X$ with
$\P^{n-1}\times\P^{1}$ and comparing with the explicit formula \ref{eq:expl formul on p n}
on $\P_{\Z}^{n}$).  The proof of Theorem \ref{thm:main toric intro}
is concluded by verifying the gap hypothesis under the conditions
in Theorem \ref{thm:main toric intro}. But we do expect that the
gap hypothesis above holds for any toric Fano variety (as explained
in \cite{a-b0} this is the toric case of a folklore conjecture). 
\end{proof}
In fact, we expect that, for any given metric $\phi,$ the supremum
\[
\sup_{\mathcal{X}}h_{\phi}(-\mathcal{K}_{\mathcal{X}})
\]
 over all integral models $\mathcal{X}$ of $X$ (with $-\mathcal{K}_{\mathcal{X}}$
relatively ample) is attained at the canonical model $\mathcal{X}$
featuring in the previous theorem. As discussed in \cite{a-b0}, this
expectation is inspired by a conjecture of Odaka \cite{o} (see \cite{h-o}
for very recent results in this direction). Translated into our setup
Odaka's conjecture suggests that, in general, the sup above is attained
at a\emph{ globally K-semistable model} $\mathcal{X}$ (i.e. all the
fibers of $\mathcal{X}\rightarrow\text{Spec \ensuremath{\Z}}$ are
K-semistable), if such model exists. In the case when $n=1$ this
expectation can be confirmed using intersection theory on arithmetic
surfaces, as shown in \cite{a-b}:
\begin{thm}
\label{thm:main conj on p one complex}Conjecture \ref{conj:height not intro}
holds when $n=1.$ In other words, for any normal model $\mathcal{X}$
of $\P^{1}$ over $\Z$ such that $\mathcal{-}\mathcal{K_{\mathcal{X}}}$
defines a relatively ample $\Q-$line bundle
\[
h_{\phi}(-\mathcal{K}_{\mathcal{X}})/2+\frac{1}{2}\text{vol\ensuremath{(-K_{\P^{1}})}}\log\int_{X}\mu_{\phi}\leq\left(1+\log\pi\right)
\]
 with equalities iff $\mathcal{X}=\P_{\Z}^{1}$ and $\phi$ is a Kähler-Einstein
metric. 
\end{thm}

\subsection{The case of diagonal hypersurfaces}

Given a positive integer $d$ and $n+2$ integers $a_{i},$ consider
the diagonal hypersurface $\mathcal{X}_{a}$ of degree $d$ in $\P_{\Z}^{n+1}$
cut out by the homogeneous polynomial
\[
\sum_{i=0}^{n+1}a_{i}x_{0}^{d},
\]
 The corresponding complex variety $X_{a}$ is Fano if and only if
$d\leq(n+1)$ and is always K-polystable (and, in particular, K-semistable.
Indeed, over $\C,$ $\mathcal{X}_{a}$ is isomorphic to the Fermat
hypersurface $\mathcal{X}_{1}$ obtained by setting $a_{i}=1$ (which
is K-semistable; see Example \ref{exa:K-st}). Moreover, $-\mathcal{K}_{\mathcal{X}_{a}}$
is defined as a relatively ample line bundle, since, by adjunction,
\begin{equation}
-\mathcal{K}_{\mathcal{X}_{a}}\simeq\mathcal{O}(n+2-d)_{|\mathcal{X}_{a}}.\label{eq:adjunction isomo}
\end{equation}
 The following result is shown in \cite{a-b}:
\begin{thm}
\label{thm:diagonal hypersurface intro}Conjecture \ref{conj:height not intro}
holds for any diagonal hypersurface $\mathcal{X}_{a}$ which is Fano
(i.e. $d\leq n+1)$. More precisely, for any volume-normalized continuous
psh metric $\phi$ on $-K_{X_{a}}$
\[
h_{\phi}(-\mathcal{K}_{\mathcal{X}_{a}})\leq h_{\phi_{FS}}(-\mathcal{K}_{\P_{\Z}^{n}})-(d-1)(n+2-d)^{n}\sum_{i=0}^{n+1}\log|a_{i}|.
\]
where $\phi_{FS}$ denotes volume-normalized metric on $-K_{\P^{n}}$corresponding
to the Fubini-Study metric on $\mathcal{O}(1).$ Moreover, the inequality
is strict if $d\geq2.$
\end{thm}

In particular, when $d\geq2$ the bound in the previous theorem improves
as the absolute values of the coefficients $a_{i}$ are increased. 

\section{\label{sec:Bounds-on-the minimal}Bounds on the minimal height of
an algebraic point }

\subsection{\label{subsec:Motivations-from-the MP }Motivations from the Manin-Peyre
conjecture }

Let $X$ be a Fano variety defined over $\Q.$ It is expected that,
after perhaps replacing $\Q$ with a finite field extension $\F$
(i.e. a number field), the number of $\F-$points $X(\F)$ are Zariski
dense \cite{b-k-s,ts}. Moreover, according to the Manin-Peyre conjecture
\cite{f-m-t,pey,pey2}, the density of $\F-$points may be quantified
as follows. Assume, for simplicity, that $X_{\Q}$ is non-singular
and fix a relatively ample model $(\mathcal{X},\mathcal{L})$ of $(X_{\Q},-K_{X_{\Q}})$
over $\Z.$ Denote by $H_{\phi}(x)$ the corresponding \emph{(exponential)
absolute height }of a given point $x\in X(\F):$ 
\begin{equation}
H_{\phi}(x)=\exp h_{\phi}(x),\,\,\,h_{\phi}(x):=-\sum_{\sigma\in G}\frac{1}{[\F:\Q]}\log\left\Vert s_{\sigma(x)}\right\Vert _{\phi},\label{eq:def of abs height}
\end{equation}
 where $G$ is the Galois group of the field extension $\Q\hookrightarrow\F$
(having fixed an embedding of $\F$ into $\C)$ and $s_{x}$ is a
generator of the free $\mathcal{O}_{\F}-$module in $-K_{X_{\F}|x}$
induced by $(\mathcal{X},\F).$ As the metric is scaled, $\phi\rightarrow\phi+\lambda,$
\begin{equation}
h_{\phi+\lambda}(x)=h_{\phi}(x)+\lambda/2,\,\,\forall\lambda\in\R.\label{eq:scaling relation of h of point}
\end{equation}

\begin{example}
\label{exa:(Naive-heights).-Let}\emph{(``Naive heights'' over $\Q+i\Q$)}.
When $\phi$ is the restriction of the Weil metric and $\F=\Q+i\Q,$the
corresponding height function $H_{\phi}(x)$ on $X(\Q+i\Q)$ is given
by the same formula as when $\F=\Q$, considered in Example \ref{exa:(height-of-weil rational}:
representing $x=(x_{0}:\cdots:x_{n})$ for integers $x_{i}$ with
no common divisors
\[
H_{\phi}(x)=\max_{i=0,...,m}\left|x_{i}\right|.
\]
(since $G$ has order two and $G(i)=-i).$ 
\end{example}

According to the\emph{ Manin-Peyre conjecture }\cite{f-m-t,pey,pey2}
there exists a thin subset $T$ of $X$ such that, setting 
\[
N(B):=\left(\sharp\left\{ x\in(X-T)(\F):\,H_{\phi}(x)\leq B\right\} \right)^{1/[\F:\Q]}
\]
the following asymptotics hold:
\[
\lim_{B\rightarrow\infty}\frac{N(B)}{B(\log B)^{r}}=\Theta(\F,\mathcal{L},\phi)
\]
 for a positive integer $r$ and a non-zero constant $\Theta(\F,\mathcal{L},\phi)$
of the following form
\[
\Theta(\F,\mathcal{L},\phi)=\eta(\F,\mathcal{L})\mu_{\phi}(X(\C))^{m_{\C}/2[\F:\Q]}\mu_{\phi}(X(\R))^{m_{\R}/[\F:\Q]}
\]
 (assuming that ``weak approximation'' holds \footnote{in general, $\Theta(\F,-\mathcal{K}_{\mathcal{X}},\phi)$ is bounded
from above by the rhs above}) where the positive constant $\eta(\F,\mathcal{L})$ only depends
on $\F$ and the model $(\mathcal{X},\mathcal{L})$ and $\mu_{\phi}(X(\C))$
and $\mu_{\phi}(X(\R))$ denote the integrals over $X(\C)$ and $X(\R)$
of the measures induced by the metric $\phi$ on $-K_{X}.$ Furthermore,
$m_{\C}$ and $m_{\R}$ denote the number of real and complex embeddings
of $\Q$ in $\F.$ In particular, since $m_{\C}+m_{\R}=[\F:\Q],$
\[
\Theta(\F,\mathcal{L},\phi+\lambda)=e^{\lambda/2}\Theta(\F,\mathcal{L},\phi)\,\,\,\forall\lambda\in\R,
\]
 which is consistent with the scaling relation \ref{eq:scaling relation of h of point}.
A conjectural explicit geometric description of the set $T$ is proposed
in \cite{l-s-t}.
\begin{rem}
In contrast to \cite{pey} we are adopting \emph{absolute} heights
here and this is why we have to take a $[\F:\Q]-$root in the definition
of $N(B)$ to get linear growth in $B.$ As a consequence, our leading
constant $\Theta(\F,\mathcal{L},\phi)$ is the $[\F:\Q]-$root  of
the leading constant $c$ appearing in \cite{pey}.
\end{rem}

\subsubsection{The minimal height of $\F-$points }

In the case when $r=0$ (which is expected to happen when the Picard
number of $X$ is one) a heuristic argument was put forth in a series
of papers by Elsenhans-Jahnel (see \cite{e-j,e-j0}) suggesting that
\begin{equation}
\min_{x\in X(\F)}H_{\phi}(x)\Theta(\F,\mathcal{L},\phi)\leq C\label{eq:ej inequality}
\end{equation}
 for a constant $C$ which might be rather insensitive to the data
$(\mathcal{X},\mathcal{L},\phi).$ The simple argument assumes, in
particular, the validity of the Manin-Peyre conjecture and goes as
follows. Assume first, as in \cite{e-j,e-j0}, that $\Q=\F$ and consider
the $N(B)$ points of $x_{i}$ of height $H_{\phi}(x_{i})\leq B.$
Under the (optimistic) hypothesis that their heights $H_{\phi}(x_{i})$
are equidistributed in $[0,B],$ in the sense that 
\begin{equation}
\min_{x\in X(\Q)-T}H_{\phi}(x)N(B)\leq CB,\label{eq:heuristic bound}
\end{equation}
the inequality \ref{eq:ej inequality} follows directly from the Manin-Peyre
conjecture. More generally, a simple scaling argument suggests that
\ref{eq:ej inequality} holds for any number field $\F.$ The case
when $X$ is a diagonal quartic threefold and cubic surface, respectively,
was considered in \cite{e-j,e-j0} (with $\F=\Q).$ The metric $\phi$
on $-K_{X}$ was taken to be the one induced by the Weil metric on
projective space (as in Example \ref{exa:(Naive-heights).-Let}).
Although extensive numerical results were provided in favor of the
inequality \ref{eq:ej inequality} (by numerically computing the left
hand side for almost a million diagonal hypersurfaces) it was shown
in \cite{e-j0,e-j} that the constant $C$ can \emph{not} be taken
uniformly over all diagonal quartic threefold, nor over all diagonal
cubic surfaces. It was, however, pointed out that a slightly weaker
inequality might still hold, which is obtained by raising $\Theta(\F,\mathcal{L},\phi)$
to the power $1/(1+\epsilon),$ for $\epsilon$ arbitrarily close
to $0.$ 
\begin{rem}
\label{rem:counter}The violation of the inequality \ref{eq:ej inequality},
exhibited in \cite{e-j0} is obtained by looking at the family of
subschemes $\mathcal{X}_{a}$ cut-out by the homogeneous polynomial
\[
-ax_{0}^{d}+x_{1}^{d}+...+x_{n+1}^{d}
\]
 on $\P_{\Z}^{n+1},$ parametrized by a positive integer $a,$ for
$d=4$ and $n=3$ (more precisely, $a$ is assumed to be a product
of distinct primes $p_{i}$ such $p_{i}\equiv3$ (mod 4). In this
case $\Theta(\Q,\mathcal{X}_{a})$ is proportional to  product over
all primes $p$ times $\mu_{\phi}(X(\R))$ and it is the contribution
from the ``bad'' primes $p_{i},$ dividing $a,$ that violates the
bound \ref{eq:ej inequality}, as $a\rightarrow\infty.$ 
\end{rem}

\subsection{\label{subsec:A-bound-on minimal}A bound on the minimal height of
an algebraic point on a K-semistable Fano variety }

In view of the counter-example recalled above and since the Manin-Peyre
conjecture can only be expected to hold for sufficiently ``large''
field extensions $\F,$ it seems natural to weaken the inequality
\ref{eq:ej inequality} by taking the infimum of both factors over
$\F:$ 
\begin{equation}
\inf_{\F}\Theta(\F,\mathcal{L},\phi)\inf_{\F}\left(\min_{x\in X(\F)}H_{\phi}(x)\right)\leq C.\label{eq:weak ej with inf F}
\end{equation}
 For a \emph{fixed} model $(\mathcal{X},\mathcal{L})$ this inequality
is clearly implied by the following one: 
\begin{equation}
\max\left\{ \mu_{\phi}(X(\C))^{1/2},\mu_{\phi}(X(\R)\right\} \left(\inf_{x\in X(\bar{\Q})}H_{\phi}(x)\right)\leq C\label{eq:weak ej inequality}
\end{equation}
(after perhaps increasing the constant $C),$ where $\bar{\Q}$ denotes
the field of algebraic points and the height $H_{\phi}(x)$ of a given
algebraic point $x$ is defined by formula \ref{eq:def of abs height},
by taking $\F$ as the residue field of the point $x.$ As next shown,
Theorem \ref{thm:arithm Vol and K semi st} implies that if $X$ is
assumed K-semistable, the complex part $X(\C)$ does satisfy the last
inequality above. More generally, consider the\emph{ $i$ th successive
minimum} $e_{i}(\mathcal{L},\phi)$ attached to $(\mathcal{L},\phi),$
introduced in \ref{eq:ej inequality} (leveraged in the proof of the
Bogomolov conjecture in \cite{zh2}): 
\[
e_{i}(\mathcal{L},\phi)=\sup_{Y_{i}}\left\{ \inf_{x\in(X-Y_{i})}h_{\phi}(x)\right\} ,
\]
where the sup ranges over all subvarieties $Y_{i}$ of $X_{\Q}$ of
codimension $i$ for $i\in[1,n+1]$ 
\begin{thm}
Let $X$ be an $n-$dimensional K-semistable Fano variety defined
over $\Q$ . Given a model $(\mathcal{X},\mathcal{L})$ of $(X_{\Q},-K_{X_{\Q}})$
over $\Z$ there exists a positive constant $C(\mathcal{L})$ such
that for any continuous psh metric $\phi$ on $-K_{X}:$ 
\[
\text{\ensuremath{\frac{1}{n+1}\sum_{i=1}^{n+1}e_{i}}(\ensuremath{\mathcal{L}},\ensuremath{\phi})}+\frac{1}{2}\log\mu_{\phi}(X(\C)\leq\log C(\mathcal{L}).
\]
In particular, 
\[
\inf_{x\in X(\bar{\Q})}H_{\phi}(x)\mu_{\phi}(X(\C))^{1/2}\leq C(\mathcal{L}).
\]
 Moreover, if the weak form of Conjecture \ref{conj:height not intro}
holds, then the constant $C(\mathcal{-}\mathcal{K}_{\mathcal{X}})$
can be taken to be of the form $e^{C_{n}/\text{vol\ensuremath{(-K_{X})}}}.$
In particular, this is the case for $\mathcal{X}$ appearing in Theorems
\ref{thm:main toric intro}, \ref{thm:main conj on p one complex}
and \ref{thm:diagonal hypersurface intro}. 
\end{thm}

\begin{proof}
By scale invariance, we may as well assume that $\mu_{\phi}(X(\C))=1.$
Hence, the first bound follows directly from combining Theorem \ref{thm:arithm Vol and K semi st}
with the lower bound on $\hat{h}(\mathcal{L},\phi)$ in Zhang's inequalities
\cite[Thm 5.2]{Zh0}, 
\begin{equation}
e_{1}(\mathcal{L},\phi)\geq\hat{h}(\mathcal{L},\phi)\geq(n+1)^{-1}\sum_{i=1}^{n\text{+1}}e_{i}(\mathcal{L},\phi).\label{eq:Zhang ine}
\end{equation}
The second bound in the corollary then follows immediately. Indeed,
the exponential of $e_{n+1}(\mathcal{L},\phi)$ coincides with the
infimum of $H_{\phi}(x)$ over $X(\bar{\Q})$ and $e_{i}(\mathcal{L},\phi)$
is decreasing in $i.$ 
\end{proof}
\begin{rem}
It is tempting to speculate that the validity of the stronger first
inequality in the previous corollary, involving the average of the
successive minima (which assumes that $X$ is K-semistable), is a
reflection of the presence of the thin set $T$ in the heuristic bound
\ref{eq:heuristic bound} (which in many situations is a subvariety
of $X_{\Q}).$ 
\end{rem}

For toric $(\mathcal{L},\phi)$ explicit formulas for the succesiva
minima are established in \cite{b-p-s}.

\section{The role of Kähler-Einstein metrics and real analogs}

Let $X$ be Fano variety defined over $\Q$ and fix a model $(\mathcal{X},\mathcal{L})$
of $(X_{\Q},-K_{X_{\Q}})$ over $\Z.$ Assume, for simplicity, that
$X$ is non-singular. 

\subsection{The case of a totally imaginary field}

Let $\F$ be a totally imaginary number field, i.e. a number field
that cannot be embedded in $\R$ (e.g. $\F=\Q+i\Q).$ Then the the
only contribution from the metric $\phi$ to Peyre's constant $\Theta(\F,\mathcal{L},\phi)$
comes from the total integral of the measure $\mu_{\phi}$ over the
complex points $X(\C).$ In this case Theorem \ref{thm:arithm Vol and K semi st}
thus says that $X$ is K-semistable iff the bound \ref{eq:ej inequality}
holds (with a constant $C(\F,\mathcal{L})$ independent of $\phi)$
when the invariant $\inf_{x\in X(\F)}H_{\phi}(x)$ is replaced by
the exponential of the normalized height $\hat{h}_{\phi}(\mathcal{L}).$
Equivalently, 
\begin{equation}
\Theta(\F,\mathcal{L},\phi)\exp\hat{h}_{\phi}(\mathcal{L})\leq C(\F,\mathcal{L})\label{eq:uniform bound on Theta wrt phi}
\end{equation}
Moreover, when $X$ is K-polystable the maximum of left hand side
above, as $\phi$ ranges over all continuous psh metrics on $L$ is
attained precisely for the Kähler-Einstein metrics on $X$ (which
are uniquely determined up to automorphisms of $X).$ Since the left
hand side is invariant under scaling of the metric, i.e. when $\phi$
is replaced by $\phi+\lambda,$ we may as well demand that the metric
$\phi$ satisfies the normalization condition that $\hat{h}_{\phi}(\mathcal{L})=0.$
Equivalently, this means that 
\begin{equation}
\hat{h}_{\phi}(\mathcal{L})=0.\label{eq:normalization h zero}
\end{equation}
In other words, under this normalization condition, the Kähler-Einstein
metrics are precisely the ones which maximize Peyre's constant $\Theta(\F,\mathcal{L},\phi).$ 

Interestingly, the normalization condition \ref{eq:normalization h zero}
appears naturally in arithmetic geometry. For example, it is satisfied
by the Weil metric on projective spaces and the Néron-Tate metric
on abelian varieties, or more generally, by the canonical metrics
$\phi$ associated to a given endomorphism $f$ of $(X,L)$ of degree
$d$ at least two (i.e. $f^{*}L=dL$ and $\phi$ is uniquely determined
by $f^{*}\phi=d\phi)$ \cite{c-s,zh0b}. The role of canonical metrics
in the context of the Manin-Peyre conjecture is, for example, discussed
in the introduction of \cite{sw}. 
\begin{rem}
The comparison of Peyre's constant $\Theta(\F,\mathcal{L},\phi)$
with the height $h_{\phi}(\mathcal{L})$ above is in line with the
suggestion put forth in \cite[Section 4.15]{ts} that $\Theta(\F,\mathcal{L},\phi)$
(or more precisely the corresponding Tamagawa number) might be interpreted
as a kind of height of $(\mathcal{X},\mathcal{L},\phi).$ However,
as pointed out in \cite[Section 4.15]{ts}, unlike the ordinary heights
Peyre's constant $\Theta(\F,\mathcal{L},\phi)$ varies in a non-trivial
fashion with respect to field extensions. This is one of the motivations
for taking the infimum over $\F$ in the bound \ref{eq:weak ej with inf F}
(but see Section \ref{subsec:Bounds-on-Peyre's} for a discussion
concerning uniform upper bounds on $\Theta(\F,\mathcal{L},\phi)).$ 
\end{rem}

\subsection{\label{subsec:The-general-case}The general case }

Now consider a general number field $\F$. Then the conjectural leading
constant $\Theta(\F,\mathcal{L},\phi)$ in the Manin-Peyre conjecture
also receives a contribution from the measure $\mu_{\phi}$ over the
\emph{real} points $X(\R).$ Since a complex Fano variety $X$ is
K-semistable iff if it is K-semistable over $\R$ \cite{zhu} it seems
thus natural to conjecture that the following real analog of Theorem
\ref{thm:arithm Vol and K semi st} holds:
\[
X\,\text{is K-semistable \ensuremath{\text{\ensuremath{\iff} \ensuremath{\sup\left\{ h_{\phi}(\mathcal{L}):\,\text{\ensuremath{\phi\,\,\text{\ensuremath{\text{cont. }}psh,}}}\int_{X(\R)}\mu_{\phi}=1\right\} <\infty} }}}
\]
By scale invariance, the volume-normalization above may be dispensed
with if $h_{\phi}(\mathcal{L})$ is replaced by the functional
\[
\phi\mapsto\hat{h}_{\phi}(\mathcal{L})+\log\int_{X(\R)}\mu_{\phi}.
\]
 Moreover, if $X$ is K-polystable it seems natural to conjecture
that the sup of this functional is attained at some metric $\phi.$
Since the K-polystability of $X$ is equivalent to the existence of
a Kähler-Einstein metric on $-K_{X}$ one would expect that the restriction
to $X(\R)$ of the maximizer is the restriction of a Kähler-Einstein
metric $\phi_{KE}$ on $X,$ invariant under complex conjugation.
If this is the case it then it follows from the fact that that $h_{\phi}(\mathcal{L})$
is increasing in $\phi$ that any maximizer of the functional is of
the form $P_{X(\R)}\phi_{KE},$ where $P_{X(\R)}\phi$ is the continuous
psh metric on $-K_{X}$ defined by the following envelope:
\[
P_{X(\R)}\phi=\sup\left\{ \psi:\text{\ensuremath{\psi\,}}\ensuremath{\text{cont. }}\text{psh \ensuremath{\text{on \ensuremath{-K_{X},}}}}\,\psi\leq\phi\,\,\text{on\,\ensuremath{X(\R)}}\right\} .
\]
(see the beginning of the proof of Theorem \ref{thm:real P one} below).
Equivalently, this means that $P_{X(\R)}\phi_{KE}$ is characterized
by the following properties \cite{b-b}: 
\[
P_{X(\R)}\phi_{KE}=\text{\ensuremath{\phi_{KE}\,\,\text{on \ensuremath{X(\R),\,\,\,(dd^{c}(P_{X(\R)}\phi_{KE}))^{n}=0\,\,\text{on \ensuremath{X-X(\R).} }}}}}
\]
More generally, if one consider any Fano variety $X$ defined over
$\R$ (but not necessarily over $\Q)$ one may expect the same conjectural
picture if $2h_{\phi}(\mathcal{L})$ is replaced by the primitive
$\mathcal{E}_{\phi_{0}}(\phi)$ of the complex Monge-Ampère operator
(formula \ref{eq:def of functional E}):
\begin{conjecture}
\label{conj:K-semi real}Let $X$ be a Fano variety defined over $\R$
and assume that $X(\R)$ is non-empty. Then 
\begin{itemize}
\item $X$ is K-semistable if and only if 
\[
\sup_{\phi}\left(\frac{1}{\text{vol}(-K_{X})}\mathcal{E}_{\psi_{0}}(\phi)+2\log\int_{X(\R)}\mu_{\phi}\right)<\infty
\]
 where $\phi$ denotes a continuous psh metric on $-K_{X}.$ 
\item If $X$ is K-polystable, the supremum above is attained precisely
when $\phi=P_{X(\R)}\phi_{KE}$ for a Kähler-Einstein metric on $-K_{X}$
which is invariant under complex conjugation.
\end{itemize}
\end{conjecture}

Likewise, it seems natural to, optimistically, propose the following
real analog of Conjecture \ref{conj:height intro}: 
\begin{conjecture}
\label{conj:real height proj}Let $\mathcal{X}$ be a projective scheme
over $\Z$ of relative dimension $n$ whose relative anti-canonical
divisor defines a relatively ample $\Q-$line bundle $-\mathcal{K}_{\mathcal{X}}$
and assume that $X$ is K-semistable and $X(\R)$ is non-empty. Let
$\phi$ be a continuous psh metric on $-K_{X}$ which is volume-normalized
over $\R,$ i.e. $\int_{X(\R)}\mu_{\phi}=1.$ Then 
\[
h_{\phi}(-\mathcal{K}_{\mathcal{X}})\leq h_{\phi_{0}}(-\mathcal{K}_{\P_{\Z}^{n}}),
\]
 where $\phi_{0}$ is volume-normalized over $\R$ and coincides with
$P_{X(\R)}\phi_{FS}$ up to scaling. Moreover, if $\mathcal{X}$ is
normal then equality holds iff $\mathcal{X}=\P_{\Z}^{n}$ and the
metric $\phi$ is Kähler-Einstein and invariant under complex conjugation. 
\end{conjecture}

Removing the normalization condition on $\phi$, the inequality in
the previous conjecture is equivalent to the following bound: 
\[
\frac{h_{\phi}(-\mathcal{K}_{\mathcal{X}})}{(n+1)}+(-K_{X})^{n}\log\int_{X(\R)}\mu_{\phi}\leq\frac{h_{P_{X(\R)}\phi_{FS}}(-\mathcal{K}_{\P_{\Z}^{n}})}{(n+1)}+(-K_{\P^{n}})^{n}\log\int_{X(\R)}\mu_{\phi_{FS}}
\]

Accordingly, as before, we will say that a \emph{weak }form of the
previous conjecture holds if there exists a constant $C_{n}$ such
that 
\[
\frac{h_{\phi}(-\mathcal{K}_{\mathcal{X}})}{(n+1)}+(-K_{X})^{n}\log\int_{X(\R)}\mu_{\phi}\leq C_{n}.
\]

We next establish the case when $n=1:$ 
\begin{thm}
\label{thm:real P one}The previous conjectures hold when $X=\P^{1}$
with its standard real structure, i.e. when $X(\R)=\P_{\R}^{1}.$
Equivalently, this means that the maximum of the following functional
on continuous psh metrics $\phi$ on $-K_{X},$ 
\[
\phi\mapsto\hat{h}_{\phi}(-\mathcal{K}_{\mathcal{X}})+\log\int_{X(\R)}\mu_{\phi},
\]
is attained precisely when $\phi=P_{X(\R)}\phi_{KE}$ for a Kähler-Einstein
metric on $-K_{X}$ which is invariant under complex conjugation.
Moreover, the maximum equals $\log2\pi$ when $\mathcal{X}=\P_{\Z}^{1}.$ 
\end{thm}

\begin{proof}
It will be enough to prove Conjecture \ref{conj:K-semi real}. Indeed,
as shown in \cite{a-b} (in the proof of Theorem \ref{thm:main conj on p one complex}),
for any fixed metric $\phi$ the functional $\mathcal{X}\mapsto h_{\phi}(-\mathcal{K}_{\mathcal{X}})$
is maximal for $\mathcal{X}=\P_{\Z}^{1}$ (and $\P_{\Z}^{1}$ is the
unique normal maximizer). First observe that, in general, to prove
Conjecture \ref{conj:K-semi real} it is enough to consider continuous
psh metrics $\phi$ having the property that $(dd^{c}\phi)^{n}$ is
supported on $X(\R).$ Indeed, if $\phi$ is any continuous psh metric,
then $P_{X(\R)}\phi$ has this property and $P_{X(\R)}\phi=\phi$
on $X(\R)$ (since $\phi$ is a contender for the sup defining $P_{X(\R)}\phi$).
Accordingly, since $\mathcal{E}_{\psi_{0}}(\phi)$ is increasing in
$\phi$ (as follows directly from the fact that its differential is
a measure, by formula \ref{eq:diff of functional E}) we get
\[
\frac{1}{\text{vol}(-K_{X})}\mathcal{E}_{\psi_{0}}(\phi)-2\log\int_{X(\R)}\mu_{\phi}\leq\frac{1}{\text{vol}(-K_{X})}\mathcal{E}_{\psi_{0}}(P_{X(\R)}\phi)-2\log\int_{X(\R)}\mu_{P_{X(\R)}\phi}
\]
 with equality iff $\mathcal{E}_{\psi_{0}}(\phi)=\mathcal{E}_{\psi_{0}}(P_{X(\R)}\phi),$
which is equivalent to $\phi=P_{X(\R)}\phi$ on all of $X$ (using
that $\phi\leq P_{X(\R)}\phi;$ see the proof of Step 2 in the proof
of Theorem 2.1 in \cite{ber0b}). 

Turning to the case when $X=\P^{1}$ we will identify $-K_{X}$ with
$\mathcal{O}(2)$ in the usual way and write $\phi=2\psi,$ where
$\psi$ is a continuous psh metric on $\mathcal{O}(1).$ Then the
functional to be maximized is 
\[
\psi\mapsto h_{\psi}(\mathcal{O}(1))+\log\int_{X(\R)}\mu_{2\psi},\,\,\,\,\left(\mu_{2\psi}=e^{-\psi}dx\right)
\]
By the previous discussion we may assume that
\begin{equation}
dd^{c}\psi=0\,\,\text{on \ensuremath{\P^{1}-\P^{1}(\R)}}.\label{eq:ddc phi vanishes}
\end{equation}
We will deduce the desired result from the \emph{Moser-Trudinger inequality}
on $S^{1}$ \cite{bech}, saying that for any smooth function $v$
on $S^{1}$ 
\begin{equation}
-\frac{1}{4\pi}\int_{\boldsymbol{D}}\left|\nabla\widetilde{v}\right|^{2}+\int v\frac{d\theta}{2\pi}+\log\int e^{-v}\frac{d\theta}{2\pi}\leq0\label{eq:MT on circle}
\end{equation}
 where $\widetilde{v}$ denotes the harmonic extension of $v$ to
the unit-disc $\boldsymbol{D}$ . Moreover, equality holds iff there
exists a constant $c$ and an element $A\in GL(2,\C),$ fixing $S^{1},$
such that $e^{-v}d\theta=A^{*}(e^{-(v+c)}d\theta).$ We will adopt
the standard embedding of $\C$ in the complex projective line $\P^{1},$
so that $S^{1}$ may be identified with a subset of $\P^{1}.$ Likewise
we will identify the closure in $\P^{1}$ of the real line $\R$ in
$\C$ with $\P^{1}(\R).$Consider the standard action of $SU(2)$
on $\C^{2},$ where an element in $SU(2)$ is identified with a matrix
$A,$ 
\[
A=\left(\begin{array}{cc}
a & -\bar{b}\\
b & \bar{a}
\end{array}\right),\,\,\,a\bar{a}+b\bar{b}=1.
\]
It induces an $SU(2)-$action on $\mathcal{O}(k)\rightarrow\P^{1}$
for any integer $k,$ which coincides with the natural action on $-K_{X}$
under the usual identification of $-K_{X}$ with $\mathcal{O}(2)$
(using that $\det A=1).$ Fix $T\in SU(2)$ mapping $\P^{1}(\R)$
onto $S^{1}$$,$ e.g. $(a,b)=(1,i)/\sqrt{2.}$ Decompose the metric
$\psi$ on $\mathcal{O}(1)$ as
\[
\psi=\phi_{FS}+u,
\]
 where $\phi_{FS}$ denotes the Fubini-Study metric on $\mathcal{O}(1)$
and $u$ is a smooth function on $\P^{1}.$ 

\emph{Step1:} 
\[
\int_{X(\R)}\mu_{2\psi}=\frac{1}{2}\int_{S^{1}}e^{-T^{*}u}d\theta.
\]
 To see this note that on the affine real line $\R$ we can express
$\mu_{2\psi}=e^{-u}e^{-\phi_{FS}(x)}dx.$ It will thus be enough to
show that 
\begin{equation}
T^{*}(e^{-\phi_{FS}(x)}dx)=\frac{1}{2}d\theta.\label{eq:T pull is dtheta}
\end{equation}
 Since the Fubini-Study metric is $SU(2)-$invariant, the measure
$e^{-\phi_{FS}(x)}dx(=(1+x^{2})^{-1}dx)$ is invariant under the subgroup
of $SU(2)$ fixing $\P^{1}(\R)$ and hence $T^{*}(e^{-\phi_{FS}(x)}dx)$
is a measure on $S^{1}$ which is invariant under the subgroup of
$SU(2)$ fixing $S^{1}.$ But the latter subgroup simply acts by rotations
on $S^{1}$ and hence the formula \ref{eq:T pull is dtheta} follows,
using that $\int_{\R}(1+x^{2})^{-1}dx=\pi$).

\emph{Step 2:} setting $v=T^{*}u$
\[
-h_{\psi}(\mathcal{O}(1))=\frac{1}{4\pi}\int_{\boldsymbol{D}}\left|\nabla\widetilde{v}\right|^{2}dxdy-\int v\frac{d\theta}{2\pi}-\log2
\]
To prove this first note that for any continuous psh metric $\psi$
on $\mathcal{O}(1)$ 
\begin{equation}
h_{\psi}(\mathcal{O}(1))=h_{T^{*}\psi}(\mathcal{O}(1)),\label{eq:invariance under T}
\end{equation}
 Indeed, by the definition \ref{eq:def of xhi L infty}, combined
with \cite[Lemma 5.3]{b-b}, the following well-known formula holds
\[
h_{\psi}(\mathcal{O}(1))/2=\sup\lim_{k\rightarrow\infty}k^{-2}\log\sup_{(\P^{1})^{N_{k}}}\left\Vert \det S\right\Vert _{\psi},\,\,\,\,N_{k}:=\dim H^{0}(\P^{1},k\mathcal{O}(1)),
\]
 where $\det S$ denotes the determinant section of the $N_{k}-$fold
tensor product of $k\mathcal{O}(1)$ induced from the standard multinomial
bases. Since $\det T=1$ the diagonal action of $T$ on $(\P^{1})^{N_{k}}$
leaves the section $\det S$ invariant and hence formula \ref{eq:invariance under T}
follows.

Now decompose $T^{*}\psi=T^{*}\phi_{FS}+v.$ Since $T^{*}\phi_{FS}=\phi_{FS}$
we have that $T^{*}\psi=v+\log2$ on $S^{1}$ in the standard trivialization
of $\mathcal{O}(2)\rightarrow\C$ (where $\phi_{FS}$ corresponds
to the function $\log(1+|z|^{2})).$ Moreover, by formula \ref{eq:height on Pn},
\[
h_{\Phi}(\mathcal{O}(1))=\mathcal{E}_{\Phi_{0}}(\Phi)=-\frac{1}{2}\int_{\P^{1}}d(\Phi-\Phi_{0})\wedge d^{c}(\Phi-\Phi_{0})+\int_{\P^{1}}(\Phi-\Phi_{0})dd^{c}\Phi_{0},
\]
where $\Phi_{0}$ denotes the Weil metric. Note that $dd^{c}\Phi_{0}$
is the invariant probability measure on $S^{1}:$ 
\[
dd^{c}\Phi_{0}=S^{1}\wedge d\theta/2\pi
\]
 Moreover, setting $\Phi=T^{*}\psi$ we have that $dd^{c}\Phi$ is
supported on $S^{1}$ (since $dd^{c}\psi$ is assumed to be supported
on $\P^{1}(\R)).$ Since $dd^{c}\Phi_{0}$ is also supported on $S^{1}$
it follows that the function $\Phi-\Phi_{0}$ is harmonic on the complement
of $S^{1}.$ Since its restriction to $S^{1}$ equals $v+\log2$ it
follows, in particular, that the restriction of $\Phi-\Phi_{0}$ to
$\D$ coincides with $\widetilde{v}+\log2.$ Hence, exploiting the
anti-holomorphic symmetry on $\P^{1}$ fixing $S^{1},$ 
\[
\int_{\P^{1}}d(\Phi-\Phi_{0})\wedge d^{c}(\Phi-\Phi_{0})=2\int_{\D}d\widetilde{v}\wedge d^{c}\widetilde{v}=2\frac{1}{4\pi}\int_{\D}\left|\nabla\widetilde{v}\right|^{2}dxdy.
\]
 All in all, this means that 
\[
2h_{\psi}(\mathcal{O}(1))=2h_{T^{*}\psi}(\mathcal{O}(1))=\frac{1}{2}2\frac{1}{4\pi}\int_{\D}\left|\nabla\widetilde{v}\right|^{2}dxdy+\int vd\theta/2\pi+\log2
\]
 which concludes the proof of Step 2. 

Finally, combining Step 1 and Step 2 reveals that the functional

\[
\phi\mapsto\hat{h}_{\phi}(-\mathcal{K}_{\P^{1}})+\log\int_{X(\R)}\mu_{\phi}
\]
 coincides, up to an additive constant, with the functional appearing
in the left hand side of formula \ref{eq:MT on circle}, which is
maximal for $v=0.$ Since $v=T^{*}u$ this means that functional of
$\phi$ above is maximal when $u=0,$ i.e. when $\phi=2\phi_{FS}.$
Finally, by the uniqueness result for maximizers of the Moser-Trudinger
inequality \ref{eq:MT on circle} the metric $\phi$ is maximal iff
there exists a real constant $c$ and $A\in GL(2,\R)$ such that $\mu_{\phi}=A^{*}\mu_{\phi_{FS}+c}$
which equivalently means that $\phi-c=A^{*}\phi_{FS}$ for some $A\in GL(2,\R)$
where $A^{*}$ denotes the pull-back under the canonical action on
$-K_{X}.$ Since any Kähler-Einstein metric on $-K_{\P^{1}}$is (up
to an additive constant) of the form $A^{*}\phi_{FS}$ for $A\in GL(2,\C)$
this means that the maximizers come from the Kähler-Einstein metrics
which are invariant under complex conjugation on $\P^{1},$as desired. 
\end{proof}

\section{Back to bounds on the height of minimal algebraic points and successive
minima}

Just as in Section \ref{subsec:A-bound-on minimal} the validity of
Conjecture \ref{conj:real height proj} would imply uniform upper
bounds on the average of the successive minima, using Zhang's inequality
\ref{eq:Zhang ine}. More generally, assuming that the weak form of
both the complex and the real versions of the conjectural height bounds
hold, Zhang's inequality yields the following bound

\[
\max\left\{ \mu_{\phi}(X(\C))^{1/2},\mu_{\phi}(X(\R)\right\} \left(\inf_{x\in X(\bar{\Q})}H_{\phi}(x)\right)\leq e^{C_{n}/\text{vol\ensuremath{(-K_{X})}}},
\]
 which, as explained in Section \ref{subsec:A-bound-on minimal} can
be viewed as weak form of the inequality \ref{eq:ej inequality} proposed
by Elsenhans-Jahnel. 

In particular, in the case when $X=\P^{1}$ combining Zhang's inequality
with Theorem \ref{thm:real P one} gives 
\begin{equation}
\frac{1}{2}\left(e_{0}(\mathcal{-K}_{\P_{\Z}^{1}},\phi)+(e_{1}(\mathcal{-K}_{\P_{\Z}^{1}},\phi)\right)\leq\hat{h}_{\phi}(\mathcal{-K}_{\P_{\Z}^{1}})\leq\log\frac{2\pi}{\int_{\P^{1}(\R)}\mu_{\phi}}.\label{eq:Zhang mt on p one}
\end{equation}
 Hence, 
\[
\frac{1}{2}\left(e_{0}(\mathcal{-K}_{\P_{\Z}^{1}},\phi)+(e_{1}(\mathcal{-K}_{\P_{\Z}^{1}},\phi)\right)\leq\log\frac{2\pi}{\int_{\P^{1}(\R)}\mu_{\phi}}
\]
 However, while $P_{X(\R)}\phi_{KE}$ gives equality in the second
inequality in formula \ref{eq:Zhang mt on p one} (the value $\log2)$
this is not so in the first one. Indeed, it follows from \cite[Thm 1]{f-p-p}
(or the general toric formulas in \cite{b-p-s}) that $P_{X(\R)}\phi_{KE}$
gives the value $2^{-1}\log2$ in the first term. Interestingly, the
situation changes when $\P^{1}(\R)$ is replaced by $S^{1},$ as next
observed. 
\begin{prop}
The following two inequalities hold
\[
\frac{1}{2}\left(e_{0}(\mathcal{-K}_{\P_{\Z}^{1}},\phi)+(e_{1}(\mathcal{-K}_{\P_{\Z}^{1}},\phi)\right)\leq\log\frac{2\pi}{\int_{S^{1}}\mu_{\phi}}
\]
\[
\inf_{\P^{1}(\bar{\boldsymbol{Q})}}H_{\phi}\leq\frac{2\pi}{\int_{S^{1}}\mu_{\phi}}
\]
 with equalities if $\phi$ is the metric on $-K_{\P^{1}}$induced
by the Weil metric. 
\end{prop}

\begin{proof}
Just as in the proof of Theorem \ref{thm:real P one}it follows from
the Moser-Trudinger inequality on $S^{1}$ that 
\[
\hat{h}_{\phi}(\mathcal{-K}_{\P_{\Z}^{1}})\leq\log\frac{2\pi}{\int_{S^{1}}\mu_{\phi}}
\]
with equality when $\phi$ is the Weil metric. But for the Weil metric
it is well-known that $0=e_{0}(\mathcal{-K}_{\P_{\Z}^{1}},\phi)=e_{1}(\mathcal{-K}_{\P_{\Z}^{1}},\phi)=\hat{h}_{\phi}(\mathcal{-K}_{\P_{\Z}^{1}})$
\cite{b-p-s} and hence the proposition follows from Zhang's inequality
\ref{eq:Zhang ine}.
\end{proof}

\subsection{\label{subsec:Bounds-on-Peyre's}Bounds on Peyre's constant and Tamagawa
measures?}

If the weak forms of the conjectures \ref{conj:height not intro}
and \ref{conj:real height proj} hold, then for any given K-semistable
Fano variety $X$ defined over $\Q$ one obtains an upper bound on
the contribution to Peyre's constant $\Theta(\F,-\mathcal{K}_{\mathcal{X}},\phi)$
coming from the metric $\phi:$
\begin{equation}
\mu_{\phi}(X(\C))^{m_{\C}/2[\F:\Q]}\mu_{\phi}(X(\R))^{m_{\R}/[\F:\Q]}\exp\hat{h}_{\phi}(\mathcal{-\mathcal{K}_{\mathcal{X}}})\leq e^{C_{n}/\text{vol\ensuremath{(-K_{X})}}}.\label{eq:bound on Peyres constant from phi}
\end{equation}
 For non-singular $X$ the right hand side above is universally bounded,
since $1/\text{vol\ensuremath{(-K_{X})}}$ is bounded from above by
a constant only depending on $n$ (as there is only a finite number
of families of non-singular Fano varieties of a given dimension).
It is thus tempting to ask if such uniform bounds hold for Peyre's
complete constant $\Theta(\F,-\mathcal{K}_{\mathcal{X}},\phi)?$In
other words, is there a constant $\theta_{n}$ such that 
\begin{equation}
\Theta(\F,-\mathcal{K}_{\mathcal{X}},\phi)\exp\hat{h}_{\phi}(\mathcal{-\mathcal{K}_{\mathcal{X}}})\leq\theta_{n}?\label{eq:question bound on Peyre}
\end{equation}
For example, this would follow if the arithmetic contribution $\eta(\F,\mathcal{-\mathcal{K}_{\mathcal{X}}})$
to Peyre's constant $\Theta(\F,-\mathcal{K}_{\mathcal{X}},\phi)$
could be be bounded by a constant only depending on $n.$ However,
$\eta(\F,\mathcal{-\mathcal{K}_{\mathcal{X}}})$ is unbounded along
the family $\mathcal{X}_{a}$ of diagonal hypersurfaces discussed
in Remark \ref{rem:counter} . \emph{If} the bound \ref{eq:question bound on Peyre}
holds, then it follows from Zhang's inequality \ref{eq:Zhang ine}
that 
\[
\inf_{x\in X(\bar{\Q})}H_{\phi}(x)\Theta(\F,-\mathcal{K}_{\mathcal{X}},\phi)\leq\theta_{n},
\]
 which can be viewed as a weaker version of the inequality \ref{eq:ej inequality}
proposed by Elsenhans-Jahnel, where the inf is taken over all of $X(\bar{\Q}),$
instead of $X(\F).$ For example, if $\mathcal{X}_{a}$ is a diagonal
hypersurface of degree $d$ of the form appearing in Remark \ref{rem:counter},
then 
\[
\inf_{x\in\mathcal{X}_{a}(\Q)}H_{\phi}(x)\gtrsim a^{1/d},\,\,\,\exp\hat{h}_{\phi}(\mathcal{-\mathcal{K}_{\mathcal{X}}})\sim a^{1/d(n+1)}
\]
 (the first bound is elementary \cite{e-j0} and the second one follows
from Mahler's bounds in formula \ref{eq:bounds on height of Weil}).
Accordingly, while the family $\mathcal{X}_{a}$ violates the inequality
\ref{eq:ej inequality} proposed by Elsenhans-Jahnel, it turns out
that it does not violate the bound \ref{eq:question bound on Peyre}
(as discussed in Section \ref{subsec:Motivation-coming-from} below).
It should be stressed that, in general, bounds on the minimal height
of a algebraic point in a given number field $\F$ 
\[
\inf_{x\in X(\F)}H_{\phi}(x)\leq B
\]
are certainly much stronger than bounds for $\bar{\Q}.$ Indeed, the
former ones are known as ``search bounds'', as they yield algorithms
to find an $\F-$point (for example, by simply going through the finite
number of points on an ambient projective space satisfying $H_{\phi}(x)\leq B)$
\cite{mas}. 

Anyhow, the question \ref{eq:question bound on Peyre} is far beyond
the author's areas of expertise and we shall just conclude by scratching
its surface. First, if one \emph{fixes }$-\mathcal{K}_{\mathcal{X}}$
to be $\mathcal{-K}_{\P_{\Z}^{1}},$ then a bound of the form \ref{eq:question bound on Peyre}
does hold, uniformly wrt $\F$ and $\phi:$
\begin{prop}
There exists a constant $c$ (that can be made explicit) such that,
for any field $\F$ and continuous psh metric $\phi$ on $-K_{\P^{1}},$
\[
\Theta(\F,\mathcal{-K}_{\P_{\Z}^{1}},\phi)\exp\left(\hat{h}_{\phi}(\mathcal{\mathcal{-K}_{\P_{\Z}}})\right)\leq c
\]
 and, as a consequence, 
\[
\inf_{x\in\P^{1}(\bar{\Q})}H_{\phi}(x)\leq\frac{c}{\Theta(\F,\mathcal{-K}_{\P_{\Z}^{1}},\phi)}.
\]
\end{prop}

\begin{proof}
Consider first the absolute height function $H_{\phi_{0}}(x)$ attached
to the Weil metric on $\mathcal{-K}_{\P_{\Z}^{n}}.$ By \cite{l-m}
the number $N(B)$ of points in $\P^{1}(F)$ satisfies the estimate
$N_{\F}(B)^{1/[\F:\Q]}/B\leq b$ for a constant $b$ independent of
$\F.$ Since the Manin-Peyre conjecture holds for the Weil metric
on $\mathcal{-K}_{\P_{\Z}^{n}}$ \cite{sc,pey} it follows that Peyre's
constant $\Theta(\F,\mathcal{-K}_{\P_{\Z}^{n}},\phi_{0})$ satisfies
\[
\Theta(\F,\mathcal{-K}_{\P_{\Z}^{1}},\phi_{0})\leq b.
\]
Next, we will use that the Manin-Peyre conjecture holds for any continuous
metric $\phi$ on $\mathcal{-K}_{\P_{\Z}^{n}}$ for an empty thin
set $T$ \cite[Cor 6.2.16]{pey}. Thus decomposing 
\[
\Theta(\F,\mathcal{-K}_{\P_{\Z}^{n}},\phi)=\Theta(\F,\mathcal{-K}_{\P_{\Z}^{n}},\phi_{0})\frac{\mu_{\phi}(X(\C))^{m_{\C}/2[\F:\Q]}\mu_{\phi}(X(\R))^{m_{\R}/[\F:\Q]}}{\mu_{\phi_{0}}(X(\C))^{m_{\C}/2[\F:\Q]}\mu_{\phi_{0}}(X(\R))^{m_{\R}/[\F:\Q]}}
\]
 and estimating 
\[
\frac{\mu_{\phi}(X(\C))^{m_{\C}/2[\F:\Q]}\mu_{\phi}(X(\R))^{m_{\R}/[\F:\Q]}}{\mu_{\phi_{0}}(X(\C))^{m_{\C}/2[\F:\Q]}\mu_{\phi_{0}}(X(\R))^{m_{\R}/[\F:\Q]}}\leq\frac{\max\left\{ \mu_{\phi}(X(\C))^{1/2},\mu_{\phi}(X(\R))\right\} }{\min\left\{ \mu_{\phi_{0}}(X(\C))^{1/2},\mu_{\phi_{0}}(X(\R))\right\} }
\]
 the proof if concluded by applying the estimates on $\mu_{\phi}(X(\C))^{1/2}$
and $\mu_{\phi}(X(\R))$ in Theorems \ref{thm:main conj on p one complex}
and \ref{thm:real P one}. 
\end{proof}
Next, consider the case when $(-\mathcal{K}_{\mathcal{X}},\phi)$
is arithmetically ample. Then, as recalled in Section \ref{subsec:Arithmetic-ampleness},
$\hat{h}_{\phi}(\mathcal{-\mathcal{K}_{\mathcal{X}}})\geq0$ and hence,
if Question \ref{eq:question bound on Peyre} can be answered affirmatively,
then Peyre's constant is uniformly bounded from above:
\begin{equation}
\Theta(\F,-\mathcal{K}_{\mathcal{X}},\phi)\leq\theta_{n}.\label{eq:uniform bound on Peyre}
\end{equation}
For example, when $\F=\Q$ such a bound indeed holds for any non-singular
Fano hypersurface $\mathcal{X}$ in $\P_{\Z}^{2}$ endowed with the
restriction of the Weil metric (assuming that $X(\Q)$ is non-empty).
Indeed, since $\mathcal{X}$ is a conic it follows from \cite{b-HB}
that there exists a constant $\theta,$ independent of $\mathcal{X},$
such that
\[
N_{X_{\Q}}(B)/B\leq\theta,
\]
 where $N(B)$ is defined as in formula \ref{eq:N X intro}. The uniform
bound \ref{eq:uniform bound on Peyre} thus follows by letting $B\rightarrow\infty$
and using that the Manin-Peyre conjecture holds for non-singular conics
for an empty thin set $T$ (see \cite{hb}). 

\subsubsection{\label{subsec:Motivation-coming-from}Motivation coming from the
product formula}

The question \ref{eq:question bound on Peyre} is motivated by the
product formula for the exponential of the height $h_{\phi}(\mathcal{L})$
(which follows, for example, from expressing $h_{\phi}(\mathcal{L})$
in terms of Deligne pairings \cite{zh1,b-e}): 
\[
\exp h_{\phi}(\mathcal{L})=\prod_{p\leq\infty}\exp h_{\phi}(\mathcal{L})_{p},
\]
over all\emph{ places} $p$ of $\Q$ i.e. all multiplicative (normalized)
absolute values $\left|\cdot\right|_{p}$ on the global field $\Q$
\cite{b-g-s}. In other words, identifying a finite place $p$ with
a prime, $\left|\cdot\right|_{p}$ is the standard non-Archimedean
$p-$adic absolute value on $\Q$ and $\left|\cdot\right|_{\infty}$
is the standard Archimedean absolute value on $\Q.$ Likewise, Peyre's
constant $\Theta(\mathcal{-K}_{\mathcal{X}},\phi)$ can also be expressed
as a product over all places $p$ of $\Q,$ under suitable assumptions.
For example, assume, for simplicity, that $X$ has Picard number one
and $\F=\Q.$ Then Peyre's constant may be expressed as follows, where
$\alpha(X)\leq1:$ 
\[
\Theta(\F,-\mathcal{K}_{\mathcal{X}},\phi)=\alpha(X)\left(\prod_{p<\infty}(1-\frac{1}{p})\mu_{p}(X(\Q_{p}))\right)\mu_{\phi}(X(\R).
\]
 (the product of the first two factors above was denoted by $\eta(\F,\mathcal{-\mathcal{K}_{\mathcal{X}}})$
in formula \ref{eq:decompos of Theta for Q}). \footnote{in general, $\Theta(\F,-\mathcal{K}_{\mathcal{X}},\phi)$ is bounded
from above by the rhs above.} In this formula $\mu_{p}$ denotes the measure on the analytic manifold
$X(\Q_{p}),$ over the local field $\Q_{p}$ of $p-$adic numbers,
determined by the model $(\mathcal{X},-\mathcal{K}_{\mathcal{X}})$
of $(X,-K_{X})$ over $\Z$ (introduced by Weil \cite{we}). Its total
mass $\mu_{p}(X(\Q_{p}))$ may be computed by counting points on $\mathcal{X}$
in the congruence class $\Z/(\Z p^{r}):$
\[
\mu_{p}(X(\Q_{p})=(p^{r})^{-n}\sharp\mathcal{X}(\frac{\Z}{\Z p^{r}})
\]
 for $r(=r(p))$ sufficiently large ($r=1$ when $p$ has good reduction,
i.e. when the structure morphism $\mathcal{X}\rightarrow\text{Spec \ensuremath{\Z}}$
is non-singular over $p$) \cite[Thm 2.14]{sa}. In view of the two
product formulas above for $\exp h_{\phi}(\mathcal{L})$ and $\Theta(\F,-\mathcal{K}_{\mathcal{X}},\phi)$
it seems thus natural to ask, as in question \ref{eq:question bound on Peyre},
if one can include the ``missing'' contribution from the finite
places $p$ in the bound \ref{eq:bound on Peyres constant from phi}?

Consider, for example, the case when $\mathcal{X}$ is a Fano hypersurface
in $\P_{\Z}^{n+1},$ of codimension one (which, in the generic case
if K-semistable, as recalled in Example \ref{exa:K-st}). Assume also
that $n\geq3,$ which ensures that $X_{\Q}$ and $X(=X_{\C})$ have
Picard number one (by Lefschetz hyperplane theorem). Denote by $S$
the set of all primes $p$ with good reduction (whose complement $S^{c}$
is finite). Then the contribution to $\Theta(\F,-\mathcal{K}_{\mathcal{X}},\phi)$
from $S$ can, in fact, be bounded by a constant only depending on
$n:$ 
\begin{equation}
\prod_{p\in S}(1-\frac{1}{p})\mu_{p}(X(\Q_{p}))\leq C_{n}\label{eq:unif bound over S}
\end{equation}
 Indeed, by Deligne's proof of the Weil conjectures, there exists
a constant $C_{n,d}$ such that, for any $n-$dimensional hypersurface
of degree $d$ in $\P_{Z}^{n+1},$ 
\[
\prod_{p\in S}(1-\frac{1}{p})p{}^{-n}\sharp\mathcal{X}(\F_{p})\leq C_{n,d},\,\,\,\F_{p}:=\frac{\Z}{\Z p}
\]
(which, when $X$ is Fano, implies the bound \ref{eq:unif bound over S},
since then $d\leq(n+1)^{n}.$ This a consequence of Deligne's bound
\cite[Thm 8.1]{del},
\[
\left|\sharp\mathcal{X}(\F_{p})-\pi_{n}\right|\leq b'_{n,d}p^{n/2}\,\,\,\,\pi_{n}:=\P_{\Z}^{n}(\F_{p})=p^{n}+p^{n-1}+...+1,
\]
where $b'_{n}$ is the $n$th primitive Betti number, which, in the
hypersurface case is bounded from above by a constant only depending
on $n$ and $d$ (see \cite[Section 3]{g-l}). 

However, the contribution from the bad primes $S^{c}$ will, in general,
not be uniformly bounded. For example, for the diagonal hypersurface
$\mathcal{X}_{a},$ parametrized by a positive integer $a,$ discussed
in Remark \ref{rem:counter}, the contribution from $S^{c}$ diverges
as  $a$ tends to infinity. On the other hand, as $a$ is increased,
the inequality \ref{eq:bound on Peyres constant from phi}, concerning
the contribution to $\Theta(\F,-\mathcal{K}_{\mathcal{X}},\phi)$
coming from $p=\infty,$ can be improved by including a negative power
of $a$ in the right hand side (see Theorem \ref{thm:diagonal hypersurface intro}).
As it turns out, this, in fact, compensates for the bad primes in
this example. Indeed, as shown in \cite[Section 2]{e-j0}, 
\[
(p^{r})^{-n}\sharp\mathcal{X}_{a}(\frac{\Z}{\Z p^{r}})=p^{-2}(p+1)^{2}\leq4
\]
 Since, the bad primes $p$ divide $da$ this means that
\[
\prod_{p\in S^{c}}(1-\frac{1}{p})\mu_{p}(X_{a}(\Q_{p}))\leq\prod_{p\vert da}4\leq c_{\epsilon}a^{\epsilon}
\]
 for any $\epsilon>0$ (using \cite[Lemma 3.1.10]{e-j0}, in the last
equality).

\subsubsection{The limit when $\F$ is increased towards $\bar{\Q}$}

To make Question \ref{eq:question bound on Peyre} more accessible
one could consider the limit when the number field $\F$ is increased
towards $\bar{\Q}$ (in the sense of nets). This leads to the following
simplified question: assuming that $X$ is K-semistable, is there
a constant $\theta_{n}$ only depending on $n$ such that
\[
\limsup_{\F\rightarrow\bar{\Q}}\Theta(\F,-\mathcal{K}_{\mathcal{X}},\phi)\exp\hat{h}_{\phi}(\mathcal{-\mathcal{K}_{\mathcal{X}}})\leq\theta_{n}?
\]
If this is the case, Zhang's inequality \ref{eq:Zhang ine}, implies
that the inequality \ref{eq:ej inequality}, proposed by Elsenhans-Jahnel,
holds in the limit $\F\rightarrow\bar{\Q},$ i.e. that 
\[
\limsup_{\F\rightarrow\bar{\Q}}\left(\Theta(\F,\mathcal{L},\phi)\min_{x\in X(\F)}H_{\phi}(x)\right)\leq\theta_{n}.
\]

\end{document}